\documentclass{amsart}
\usepackage{amsmath}
\usepackage{amssymb}
\usepackage[all]{xy}

\RequirePackage{color}
\definecolor{myred}{rgb}{0.75,0,0}
\definecolor{mygreen}{rgb}{0,0.5,0}
\definecolor{myblue}{rgb}{0,0,0.65}

\RequirePackage{ifpdf}
\ifpdf
  \IfFileExists{pdfsync.sty}{\RequirePackage{pdfsync}}{}
  \RequirePackage[pdftex,
   colorlinks = true,
   urlcolor = myblue, 
   citecolor = mygreen, 
   linkcolor = myblue, 
   pdfusetitle,
   pdfauthor = {{Anthony Henderson}},
   pdfpagemode=None,
   bookmarksopen=true]{hyperref}
\else
  \RequirePackage[hypertex]{hyperref}
\fi

\let\oldeqn\equation
\let\endoldeqn\endequation
\renewenvironment{equation}{\oldeqn\setlength{\thickmuskip}{10mu plus 5mu}}{\endoldeqn\relax}
\catcode`\*11
\let\oldeqn*\equation*
\let\endoldeqn*\endequation*
\renewenvironment{equation*}{\oldeqn*\setlength{\thickmuskip}{10mu plus 5mu}}{\endoldeqn*\relax}
\catcode`\*12

\newcommand{\C}{\mathbb{C}}
\newcommand{\Z}{\mathbb{Z}}
\newcommand{\N}{\mathbb{N}}
\newcommand{\R}{\mathbb{R}}
\newcommand{\Aa}{\mathbb{A}}
\newcommand{\Pp}{\mathbb{P}}
\newcommand{\Gg}{\mathbb{G}}

\newcommand{\GL}{\mathrm{GL}}
\newcommand{\SL}{\mathrm{SL}}
\newcommand{\SU}{\mathrm{SU}}

\newcommand{\Mat}{\mathrm{Mat}}

\newcommand{\fg}{\mathfrak{g}}

\newcommand{\fsl}{\mathfrak{sl}}
\newcommand{\fso}{\mathfrak{so}}
\newcommand{\fsp}{\mathfrak{sp}}

\newcommand{\A}{\mathrm{A}}
\newcommand{\D}{\mathrm{D}}
\newcommand{\E}{\mathrm{E}}

\newcommand{\cN}{\mathcal{N}}
\newcommand{\cO}{{\mathcal{O}}}

\newcommand{\Gr}{\mathsf{Gr}}
\newcommand{\Grom}{\mathsf{Gr}_{0}}
\newcommand{\Grbar}{\overline{\mathsf{Gr}}}

\newcommand{\fM}{\mathfrak{M}}

\newcommand{\bv}{\mathbf{v}}
\newcommand{\bw}{\mathbf{w}}
\newcommand{\bi}{\mathbf{i}}
\newcommand{\bj}{\mathbf{j}}

\newcommand{\Bun}{\mathrm{Bun}}

\DeclareMathOperator{\Hom}{Hom}

\newcommand{\simto}{\mathrel{\overset{\sim}{\to}}}

\newcommand{\cD}{\mathcal{D}}
\newcommand{\cS}{\mathcal{S}}
\newcommand{\cE}{\mathcal{E}}
\newcommand{\cF}{\mathcal{F}}
\newcommand{\cU}{\mathcal{U}}

\numberwithin{equation}{section}
\newtheorem{thm}{Theorem}[section]
\newtheorem{lem}[thm]{Lemma}
\newtheorem{prop}[thm]{Proposition}

\theoremstyle{definition}

\theoremstyle{remark}
\newtheorem{rmk}[thm]{Remark}
\newtheorem{ex}[thm]{Example}

\title[Involutions on the affine Grassmannian]{Involutions on the affine Grassmannian and moduli spaces of principal bundles}

\author{Anthony Henderson}
\address{School of Mathematics and Statistics\\
  University of Sydney, NSW 2006\\
  Australia}
\email{anthony.henderson@sydney.edu.au}


\thanks{The author's research was supported by ARC Future Fellowship No.~FT110100504.}

\begin{document}

\begin{abstract}
Let $G$ be a simply connected semisimple group over $\mathbb{C}$. We show that a certain involution of an open subset of the affine Grassmannian of $G$, defined previously by Achar and the author, corresponds to the action of the nontrivial Weyl group element of $\SL(2)$ on the framed moduli space of $\mathbb{G}_m$-equivariant principal $G$-bundles on $\mathbb{P}^2$. As a result, the fixed-point set of the involution can be partitioned into strata indexed by conjugacy classes of homomorphisms $N\to G$ where $N$ is the normalizer of $\mathbb{G}_m$ in $\mathrm{SL}(2)$. When $G=\mathrm{SL}(r)$, the strata are Nakajima quiver varieties $\mathfrak{M}_0^{\mathrm{reg}}(\mathbf{v},\mathbf{w})$ of type D.   
\end{abstract}

\maketitle


\section{Introduction}
\label{sect:intro}


\subsection{Outline}
\label{ss:outline}

Let $G$ be a simply connected semisimple algebraic group over $\C$. The ind-group $G[z]:=G(\C[z])$ consists of variety morphisms $g(z):\Gg_a\to G$, where $z$ represents the coordinate on $\Gg_a\cong\Aa^1$. We obtain three closed sub-ind-varieties of $G[z]$ by imposing conditions related to the group structures on $\Gg_a$ and $G$:
\begin{itemize}
\item Requiring that $g(0)=1$, we obtain the first congruence subgroup $G[z]_1$.
\item Requiring that $g(0)=1$ and $g(-z)=g(z)^{-1}$, we obtain, say, $Y\subset G[z]_1$.
\item Requiring that $g(z)$ is a group homomorphism, we obtain a copy of the nilpotent cone $\cN$ of $G$, via the embedding $e:\cN\hookrightarrow Y:x\mapsto \exp(xz)$. 
\end{itemize}

It is known that $G[z]_1$ and $\cN$ parametrize suitable kinds of principal $G$-bundles on a rational surface. The main result of this article provides an analogous interpretation for $Y$, forming the middle vertical arrow of the commutative diagram:
\begin{equation}\label{eqn:intro-diag}
\vcenter{\xymatrix@C=20pt{
\ \cN\  \ar@{^{(}->}[r]^-{e} \ar@{<-}[d] & \ Y\  \ar@{^{(}->}[r] \ar@{<-}[d] & G[z]_1 \ar@{<-}[d]_(.5){\Psi}\\
\Bun_G(\Aa^2/\SL(2))\, \ar@{^{(}->}[r] & \Bun_G(\Aa^2/N)\, \ar@{^{(}->}[r] & \Bun_G(\Aa^2/\Gg_m)
}}
\end{equation}
Here, following the notation of~\cite{bf}, for any subgroup $\Gamma$ of $\SL(2)$, $\Bun_{G}(\Aa^2/\Gamma)$ denotes the moduli space of $\Gamma$-equivariant principal $G$-bundles on $\Pp^2$ equipped with a trivialization on the line at infinity $\Pp^2\smallsetminus\Aa^2$. The left-hand vertical arrow in~\eqref{eqn:intro-diag} refers to Kronheimer's result~\cite[Theorem 1]{kron} concerning the case where $\Gamma=\SL(2)$ (equivalently, by~\cite{donaldson}, the case of `$\SU(2)$-equivariant instantons on $\R^4$'). The right-hand vertical arrow refers to the result of Braverman and Finkelberg~\cite[Theorem 5.2]{bf} concerning the case where $\Gamma$ is the diagonal subgroup $\Gg_m$ of $\SL(2)$ (`$S^1$-equivariant instantons'). For the middle vertical arrow we take $\Gamma$ to be
\[
N:=N_{\SL(2)}(\Gg_m)=\langle\Gg_m,[\begin{smallmatrix}0&1\\-1&0\end{smallmatrix}]\rangle.
\]  

The vertical arrows in~\eqref{eqn:intro-diag} are $G$-equivariant bijective (ind-)variety morphisms which induce variety isomorphisms between corresponding strata. The relevant $G$-stable stratifications on the (ind-)varieties in~\eqref{eqn:intro-diag} will be recalled or introduced below. A unifying principle is that in the column corresponding to $\Gamma\subseteq\SL(2)$, the strata are indexed by certain $G$-conjugacy classes of homomorphisms $\Gamma\to G$.

\subsection{The main theorem}

Henceforth we will set $z=t^{-1}$ and identify $G[t^{-1}]_1$ with the `big opposite Bruhat cell' in the affine Grassmannian $\Gr$ of $G$, namely
\[
\Grom:=G[t^{-1}]_1G[t]/G[t]\subset\Gr:=G[t,t^{-1}]/G[t].
\]
Here $G[t,t^{-1}]$ is short for $G(\C[t,t^{-1}])$, and $G[t]$ and other such notation used hereafter should be interpreted similarly.

In~\cite{ah}, Achar and the author defined an involution $\iota$ of $\Grom=G[t^{-1}]_1$ by
\begin{equation}
\iota(g(t^{-1}))=g(-t^{-1})^{-1}.
\end{equation}
The fixed-point set $(\Grom)^\iota$ of this involution is the ind-variety named $Y$ above.

Recall that $\Gr$ is the disjoint union of certain irreducible quasiprojective varieties $\Gr^\lambda$ (the `Schubert cells'), where $\lambda$ runs over the set $\Lambda^+$ of dominant coweights of $G$.
Consequently, $\Grom$ is the disjoint union of the irreducible quasi-affine varieties
\[ \Gr_0^\lambda:=\Gr^\lambda\cap\Grom. \] 
This is the stratification of $\Gr_0$ mentioned in \S\ref{ss:outline}.

It was shown in~\cite[Lemma 2.2]{ah} that $\iota(\Gr_0^\lambda)=\Gr_0^{-w_0\lambda}$, where $w_0$ denotes the longest element of the Weyl group. Hence $Y=(\Grom)^\iota$ is the disjoint union of the varieties $(\Gr_0^\lambda)^\iota$, where $\lambda$ runs over the set
\[
\Lambda_1^+:=\{\lambda\in\Lambda^+\,|\,w_0\lambda=-\lambda\}.
\]
For general $\lambda\in\Lambda_1^+$, $(\Gr_0^\lambda)^\iota$ may be disconnected, and it is a natural problem to describe the connected components. 

As mentioned above, Braverman and Finkelberg~\cite[Theorem 5.2]{bf} defined a $G$-equivariant bijection 
\begin{equation} 
\Psi:\Bun_{G}(\Aa^2/\Gg_m)\to\Grom
\end{equation}
which restricts, for each $\lambda\in\Lambda^+$, to an isomorphism of varieties
\begin{equation} \label{eqn:bf-isom}
\Bun_{G}^\lambda(\Aa^2/\Gg_m)\simto\Gr_0^\lambda,
\end{equation} 
where $\Bun_G^\lambda(\Aa^2/\Gg_m)$ denotes the connected component of $\Bun_{G}(\Aa^2/\Gg_m)$ defined by requiring that the action of $\Gg_m$ on the fibre of the $G$-bundle at the origin of $\Aa^2$ gives rise to a homomorphism $\Gg_m\to G$ in the $G$-conjugacy class of $\lambda$.

The moduli space $\Bun_{G}(\Aa^2/\Gg_m)$ carries a natural action of $N_{\GL(2)}(\Gg_m)$. Our main result identifies the corresponding action on $\Grom=G[t^{-1}]_1$, and gives a new explanation of the significance of the involution $\iota$: namely, it corresponds to the action of the nontrivial Weyl group element of $\SL(2)$. 
\begin{thm} \label{thm:normalizer}
Under the bijection $\Psi:\Bun_{G}(\Aa^2/\Gg_m)\to\Grom$,
\begin{enumerate}
\item for any $\alpha,\beta\in\C^\times$, $[\begin{smallmatrix}\alpha&0\\0&\beta\end{smallmatrix}]\in\GL(2)$ acts on $\Grom$ by $g(t^{-1})\mapsto g(\alpha\beta t^{-1})$;
\item the element $[\begin{smallmatrix}0&1\\1&0\end{smallmatrix}]\in\GL(2)$ acts on $\Grom$ by $g(t^{-1})\mapsto g(t^{-1})^{-1}$.
\end{enumerate}
Consequently, the element $[\begin{smallmatrix}0&1\\-1&0\end{smallmatrix}]\in\SL(2)$ acts on $\Grom$ via the involution $\iota$.
\end{thm}
\noindent
Part (1) of Theorem~\ref{thm:normalizer} is an easy consequence of the definition of $\Psi$ in~\cite{bf}, but part (2) is not so obvious, because that definition breaks the symmetry between the horizontal and vertical directions in $\Aa^2$. To prove Theorem~\ref{thm:normalizer}, we will reformulate the definition of $\Psi$ using explicit descriptions of $G$-bundles via transition functions.

\begin{rmk} \label{rmk:slice}
The statement of Braverman and Finkelberg~\cite[Theorem 5.2]{bf} is more general than we have recalled above: they actually showed an isomorphism between a more general subvariety $\Gr_\mu^\lambda$ in the affine Grassmannian, where $\mu$ is not necessarily the zero coweight, and a moduli space $\Bun_{G,\mu}^\lambda(\Aa^2/\Gg_m)$ defined using the embedding $\Gg_m\hookrightarrow G\times\SL(2)$ where the projection onto the $G$ factor is $\mu$. In the case where $w_0\mu=-\mu$, one could ask for an analogue of Theorem~\ref{thm:normalizer}(2) describing the action of the non-identity component of $N_{G\times\GL(2)}(\Gg_m)$. We have not investigated this.
\end{rmk}

From Theorem~\ref{thm:normalizer} we deduce moduli-space interpretations of $(\Grom)^\iota$ and $(\Gr_0^\lambda)^\iota$ in terms of the group $N=N_{\SL(2)}(\Gg_m)=\langle\Gg_m,[\begin{smallmatrix}0&1\\-1&0\end{smallmatrix}]\rangle$, and a decomposition of $(\Gr_0^\lambda)^\iota$ as a disconnected union.
\begin{thm} \label{thm:normalizer2}
The bijection $\Psi$ restricts to a bijection 
\[ \Bun_G(\Aa^2/N)\to(\Grom)^\iota. \] 
For any $\lambda\in\Lambda^+_1$, $\Psi$ restricts to an isomorphism of varieties
\[ \Bun_{G}^\lambda(\Aa^2/N)\simto(\Gr_0^\lambda)^\iota, \] 
where $\Bun_{G}^\lambda(\Aa^2/N):=\Bun_G(\Aa^2/N)\cap\Bun_G^\lambda(\Aa^2/\Gg_m)$. This isomorphism matches up the corresponding pieces of disconnected union decompositions on both sides:
\[ \Bun_G^\lambda(\Aa^2/N)=\coprod_{\xi\in\Xi(\lambda)} \Bun_{G}^\xi(\Aa^2/N)\quad \text{ and }\quad 
(\Gr_0^\lambda)^\iota=\coprod_{\xi\in\Xi(\lambda)} (\Gr_0)^{\iota,\xi}, \]
where $\Xi(\lambda)$ is the set of $G$-conjugacy classes of homomorphisms $N\to G$ whose restriction to $\Gg_m$ is $G$-conjugate to $\lambda$.
\end{thm}
\noindent
Here, and throughout the paper, we use the symbol $\coprod$ for a disconnected union of varieties, in contrast to the symbol $\bigsqcup$ which merely denotes a disjoint union.

The definition of $\Bun_G^\xi(\Aa^2/N)$ is analogous to that of $\Bun_G^\lambda(\Aa^2/\Gg_m)$: namely, it is the subvariety of $\Bun_G(\Aa^2/N)$ defined by requiring that the action of $N$ on the fibre of the $G$-bundle at the origin in $\Aa^2$ gives rise to a homomorphism $N\to G$ in the $G$-conjugacy class $\xi$. See Proposition~\ref{prop:sigma} for a definition of $(\Gr_0)^{\iota,\xi}=\Psi(\Bun_G^\xi(\Aa^2/N))$ that is independent of moduli spaces.

Note that $(\Gr_0)^\iota$ is the disjoint union of the varieties $(\Gr_0)^{\iota,\xi}$ as $\xi$ runs over the set $\Xi=\bigsqcup_{\lambda\in\Lambda_1^+}\Xi(\lambda)$ of $G$-conjugacy classes of homomorphisms $N\to G$. This is the stratification of $(\Gr_0)^\iota$ mentioned in \S\ref{ss:outline}. For general $G$, we do not know for which $\xi\in\Xi$ the varieties $(\Gr_0)^{\iota,\xi}$ are nonempty, or whether they can be disconnected.

\subsection{ADHM descriptions}

When $G$ is a classical group, we have the Atiyah--Drinfel'd--Hitchin--Manin (ADHM) description~\cite{ahdm,barth,donaldson} of $\Bun_G(\Aa^2/\Gamma)$, developed further by Nakajima (see~\cite{nak1,nak2,nak3} for $\SL(r)$ and~\cite[Appendix A]{nak4} for other classical groups). Thus Theorem~\ref{thm:normalizer2} gives rise to such a description of $(\Gr_0)^{\iota,\xi}$.

In particular, consider the case where $G=\SL(r)$, so that principal $G$-bundles are equivalent to rank-$r$ vector bundles with trivial determinant bundle. In this case, for any $\lambda\in\Lambda^+$, the ADHM description of $\Bun_{G}^\lambda(\Aa^2/\Gg_m)$ identifies it with a Nakajima quiver variety $\fM_0^{\mathrm{reg}}(\bv,\bw)$ of type $\A$; see \S\ref{ss:Gm-case}. As explained by Nakajima in~\cite[Section 2(v)]{nak4}, one should think of the underlying graph as being the McKay graph of the subgroup $\Gg_m\subset\SL(2)$, hence of type $\A_{\infty}$ (infinite in both directions), although the dimension vectors $\bv$ and $\bw$ are zero outside a finite set of vertices. Combining this with the result of Braverman and Finkelberg, one recovers the result of Mirkovi\'c and Vybornov~\cite{mvy-cr} that the varieties $\Gr_0^\lambda$ for $G=\SL(r)$ are isomorphic to quiver varieties of type $\A$; see Proposition~\ref{prop:mv}. (This works for more general $\Gr_\mu^\lambda$ as in Remark~\ref{rmk:slice}.)   

Applying similar reasoning to our Theorem~\ref{thm:normalizer2}, we obtain:

\begin{thm} \label{thm:intro-glr}
Suppose $G=\SL(r)$. Then the nonempty varieties $(\Gr_0)^{\iota,\xi}$ are connected, and are isomorphic to certain Nakajima quiver varieties $\fM_0^{\mathrm{reg}}(\bv,\bw)$ of type $\D$, described explicitly in Proposition~\ref{prop:glr}.
\end{thm}

\noindent
Here the graph underlying the quiver varieties is really the McKay graph of the subgroup $N\subset\SL(2)$, which is of type $\D_{\infty}$; see \S\ref{ss:subgroups}.

In the case where $G=\SL(2)$, $\Lambda_1^+=\Lambda^+=\{m\alpha\,|\,m\in\N\}$ where $\alpha$ denotes the positive coroot of $G$. It is known that $\Gr_0^{m\alpha}$ is isomorphic to the intersection $\cO_{(2m)}\cap\cS_{(m,m)}$ where $\cO_{(2m)}$ denotes the regular nilpotent orbit in $\fsl(2m)$ and $\cS_{(m,m)}$ denotes the Slodowy slice to an orbit of Jordan type $(m,m)$ in $\fsl(2m)$; see \S\ref{ss:proof-gl2}. 

\begin{thm} \label{thm:intro-gl2}
Suppose $G=\SL(2)$. If $m>0$ is even, $(\Gr_0^{m\alpha})^\iota$ is empty; and if $m$ is odd,
\[
(\Gr_0^{m\alpha})^\iota \cong \cO_{(2m)}^{\mathrm{C}_m}\cap\cS_{(m,m)}^{\mathrm{C}_m},
\]
where $\cO_{(2m)}^{\mathrm{C}_m}$ denotes the regular nilpotent orbit in the symplectic Lie algebra $\fsp(2m)$, and $\cS_{(m,m)}^{\mathrm{C}_m}$ denotes the Slodowy slice to an orbit of Jordan type $(m,m)$ in $\fsp(2m)$.
\end{thm} 

\subsection{Structure of the paper}

Section~\ref{sect:recollections} contains no new results. We repeat some standard definitions concerning the affine Grassmannian in \S\ref{ss:gr}, recall some of the relations between the affine Grassmannian and the nilpotent cone in \S\ref{ss:achar}, and introduce in \S\ref{ss:sl2} the example of $G=\SL(2)$ which will recur throughout.

In Section~\ref{sect:moduli} we revisit the result of Braverman and Finkelberg~\cite[Theorem 5.2]{bf} concerning $\Bun_G(\Aa^2/\Gg_m)$, reformulating their proof in somewhat more elementary terms in \S\!\S\ref{ss:p1}--\ref{ss:gm-equiv}. This enables a simple proof of Theorem~\ref{thm:normalizer} in \S\ref{ss:proof}. We then deduce Theorem~\ref{thm:normalizer2} in \S\ref{ss:n-equiv}, before setting up some further notation in \S\ref{ss:p2}.

Section~\ref{sec:glr} considers the case $G=\SL(r)$, where the moduli spaces $\Bun_G(\Aa^2/\Gamma)$ have alternative descriptions in terms of quiver varieties. We first recall the fundamental $\Gamma=\{1\}$ case in \S\ref{ss:1-case}, then discuss the McKay correspondence for reductive subgroups $\Gamma$ of $\SL(2)$ in \S\ref{ss:subgroups} and (a special case of) the definition of quiver varieties in \S\ref{ss:quiver-varieties}, before returning to the study of $\Bun_G(\Aa^2/\Gamma)$ in \S\ref{ss:gamma-equiv}. Then we consider the $\Gamma=\Gg_m$ case in detail in \S\ref{ss:Gm-case}, a necessary preliminary for the $\Gamma=N$ case in \S\ref{ss:N-case}, which includes the proof of Theorem~\ref{thm:intro-glr} in Proposition~\ref{prop:glr}. Finally, \S\ref{ss:proof-gl2} gives the proof of Theorem~\ref{thm:intro-gl2}.   

\subsection{Acknowledgements}

The original motivation for this paper was a suggestion of H.~Nakajima, following the author's seminar at Kyoto University in January 2015, that the results of~\cite{ah} should have a moduli-space interpretation via~\cite{bf,kron} (this goal is not reached in the present paper). The precise statement of Theorem~\ref{thm:normalizer} was worked out in discussions with M.~Finkelberg at the Mathematisches Forschungsinstitut Oberwolfach in May 2015, and much of the rest of the paper was written while the author was visiting A.~Licata at the Australian National University in July--September 2015. All these mathematicians and institutions are gratefully acknowledged. I am also indebted to P.~Achar, I.~Grojnowski, J.~Kamnitzer, S.~Kumar, B.~Webster and O.~Yacobi for enlightening conversations about these and related matters.  


\section{Recollections}
\label{sect:recollections}


\subsection{Subvarieties of the affine Grassmannian}
\label{ss:gr}

As in the introduction, let $G$ be a simply connected semisimple algebraic group over $\C$, and let 
\[ \Gr:=G[t,t^{-1}]/G[t] \] 
be its affine Grassmannian, an ind-variety on which $G[t,t^{-1}]$ acts transitively. See~\cite[Proposition 4.6]{laszlosorger} for the reducedness and irreducibility of $\Gr$. 

Fix a maximal torus $T$ and a Borel subgroup $B$ of $G$ with $T\subset B$. Let $\Lambda=\Hom(\Gg_m,T)$ be the coweight lattice of $G$ (written additively, as usual), and let $\Lambda^+\subset\Lambda$ be the set of dominant coweights relative to $B$. The Weyl group $W$ of $(G,T)$ acts on $\Lambda$, and each $W$-orbit has a unique representative in $\Lambda^+$. Since every homomorphism $\Gg_m\to G$ is $G$-conjugate to one whose image lies in $T$, and two coweights $\Gg_m\to T$ are $G$-conjugate if and only if they lie in the same $W$-orbit, we can think of $\Lambda^+$ as parametrizing the $G$-conjugacy classes of homomorphisms $\Gg_m\to G$. Recall that if $\lambda\in\Lambda^+$, then the unique representative in $\Lambda^+$ of the $W$-orbit of $-\lambda$ is $-w_0\lambda$, where $w_0$ denotes the longest element of $W$. In particular, the condition $w_0\lambda=-\lambda$ is equivalent to saying that $\lambda$ and $-\lambda$ are $G$-conjugate.

Let $R\subset\Lambda$ denote the set of coroots of $G$, and $R^+\subset R$ the subset of positive coroots relative to $B$. Since $G$ is assumed to be simply connected, $R$ spans $\Lambda$. We have the usual partial order $\leq$ on $\Lambda$ defined by 
\[ \mu\leq\lambda \Longleftrightarrow \lambda-\mu\in\N R^+, \]
and it is well known that $\lambda\geq 0$ for all $\lambda\in\Lambda^+$.

For any $\lambda\in\Lambda$, let $t^\lambda$ be the element of $T[t,t^{-1}]$ corresponding to $\lambda:\Gg_m\to T$, and denote the corresponding $G[t]$-orbit in the affine Grassmannian by 
\[ \Gr^\lambda:=G[t]\,t^\lambda\, G[t]/G[t]\subset\Gr. \] 
It is well known that each $\Gr^\lambda$ is a smooth irreducible quasiprojective variety, and that $\Gr^{\lambda}=\Gr^\mu$ if and only if $\lambda$ and $\mu$ belong to the same $W$-orbit. Hence
\begin{equation}
\Gr=\bigsqcup_{\lambda\in\Lambda^+}\Gr^\lambda.
\end{equation}
For $\lambda\in\Lambda^+$, let $\Grbar^\lambda$ denote the closure of $\Gr^\lambda$ in $\Gr$, an irreducible projective variety (these are the `Schubert varieties' in $\Gr$). We have
\begin{equation}
\Grbar^\lambda=\bigsqcup_{\substack{\mu\in\Lambda^+\\\mu\leq\lambda}}\Gr^\mu,
\end{equation}
where the union on the right-hand side is finite.

Recall that $G[t^{-1}]_1$ denotes the first congruence subgroup of $G[t^{-1}]$, i.e.\ the kernel of the homomorphism $G[t^{-1}]\to G$ setting $t^{-1}$ to $0$. For $\mu\in\Lambda^+$, define
\[ \Gr_\mu:=G[t^{-1}]_1\, t^{w_0\mu}\, G[t]/G[t]\subset\Gr. \] 
The following intersections are nonempty exactly when $\mu\leq\lambda$:
\[
\Gr_\mu^\lambda:= \Gr_\mu\cap\Gr^\lambda,\qquad \Grbar_\mu^\lambda:=\Gr_\mu\cap\Grbar^\lambda.
\]
In fact, $\Grbar_\mu^\lambda$ is a transverse slice to $\Gr^\mu$ inside $\Grbar^\lambda$; see~\cite[Lemma 2.9]{bf}. Note that we are following the convention and notation of~\cite{kwwy} for these transverse slices: in~\cite{bf}, a variety isomorphic to what we have called $\Gr_\mu^\lambda$ is written $\mathcal{W}_{G,\mu}^\lambda$.

In this paper we focus on the case $\mu=0$. 
We have
\begin{equation}
\Gr_0 = G[t^{-1}]_1 G[t]/G[t] = \bigsqcup_{\lambda\in\Lambda^+}\Gr_0^\lambda,
\end{equation}
where $\Gr_0^\lambda$ is an open subvariety of $\Gr^\lambda$. Note that the $G$-action on $\Gr$ preserves $\Gr_0$ and each $\Gr_0^\lambda$. For general $\lambda\in\Lambda^+$, $\Gr_0^\lambda$ consists of infinitely many $G$-orbits. 

For any $\lambda\in\Lambda^+$ we have
\begin{equation}
\Grbar_0^\lambda=\bigsqcup_{\substack{\mu\in\Lambda^+\\\mu\leq\lambda}}\Gr_0^\mu.
\end{equation} 
Note that $\Grbar_0^\lambda$ is the closure of $\Gr_0^\lambda$ in $\Gr_0$. 

As in the introduction, we identify $\Gr_0$ with $G[t^{-1}]_1$ via the isomorphism 
\begin{equation}
G[t^{-1}]_1\simto\Grom:g(t^{-1})\mapsto g(t^{-1})\,G[t]/G[t].
\end{equation}
This isomorphism is $G$-equivariant, where $G$ acts on $G[t^{-1}]_1$ by conjugation. Since $G[t^{-1}]_1$ is an affine ind-variety, so is $\Gr_0$, and each $\Grbar_0^\lambda$ is an affine variety. 


\subsection{Relations to nilpotent orbits}
\label{ss:achar}

Let $\fg$ be the Lie algebra of $G$, and $\cN\subset\fg$ the nilpotent cone. Recall that the adjoint action of $G$ on $\fg$ preserves $\cN$, and $\cN$ is the union of finitely many $G$-orbits.

In~\cite{ah}, Achar and the author studied the relationship between these nilpotent orbits and (some of) the varieties $\Gr_0^\lambda$. This involves two main morphisms: the involution $\iota:\Grom\to\Grom$ defined in the introduction, and the morphism 
\[ \pi:\Grom=G[t^{-1}]_1\to\fg \] 
obtained by viewing $\fg$ as $\ker(G(\C[t^{-1}]/(t^{-2}))\to G)$. More concretely, if we regard $G$ as a closed subgroup of some $\GL(n)$ so that elements of $G[t^{-1}]_1$ can be written $1+x_1t^{-1}+x_2t^{-2}+\cdots+x_mt^{-m}$ where $x_i\in\Mat(n)$, then $\pi$ is defined by
\begin{equation}
\pi(1+x_1t^{-1}+x_2t^{-2}+\cdots+x_mt^{-m})=x_1\ \in\fg\subseteq\Mat(n).
\end{equation}
Note that by definition of $\iota$, we have $\pi\circ\iota=\pi$. Both $\iota$ and $\pi$ are $G$-equivariant.

One of the main results of~\cite{ah} states that for $\lambda\in\Lambda^+$, $\Gr_0^\lambda$ consists of finitely many $G$-orbits if and only if $\pi(\Gr_0^\lambda)\subset\cN$, and that moreover these equivalent conditions hold precisely when $\lambda$ is a \emph{small} coweight, meaning that there is no $\alpha\in\Lambda^+\cap R$ such that $2\alpha\leq\lambda$.

As mentioned in the introduction, \cite[Lemma 2.2]{ah} showed that 
\begin{equation}
\iota(\Gr_0^\lambda)=\Gr_0^{-\lambda}=\Gr_0^{-w_0\lambda}\ \text{ for all }\lambda\in\Lambda^+.
\end{equation} Write $\Lambda^+_1$ for the set of $\lambda\in\Lambda^+$ satisfying $w_0\lambda=-\lambda$. Thus, when $\lambda\in\Lambda^+_1$, one has an induced involution $\iota$ of $\Gr_0^\lambda$ and can consider the fixed-point subvariety $(\Gr_0^\lambda)^\iota$, which is possibly disconnected. Note that $(\Gr_0^\lambda)^\iota$ is $G$-stable.

As was also mentioned in the introduction, we have a canonical $G$-equivariant embedding
\[ e:\cN\hookrightarrow(\Grom)^\iota:x\mapsto\exp(xt^{-1}), \]
where $\exp:\fg\to G$ is the exponential map (whose restriction to $\cN$ is a variety morphism).
Note that $e$ is a section of $\pi$ in the sense that $\pi(e(x))=x$.

\begin{rmk} 
The embedding $e$ is not to be confused with the \emph{Lusztig embeddings} 
\begin{equation} \label{eqn:embeddings}
\begin{split}
&\cN\hookrightarrow\Grom:x\mapsto 1+xt^{-1}\text{ and}\\ 
&\cN\hookrightarrow\Grom:x\mapsto (1-xt^{-1})^{-1}, 
\end{split}
\end{equation}
which make sense only when $G=\SL(n)$. The embeddings in~\eqref{eqn:embeddings} are sections of $\pi$ which are interchanged by $\iota$; they were introduced in~\cite{lusztig} (see also~\cite[Section 4.1]{ah}).   
\end{rmk}

For the nilpotent cone $\cN$, the stratification mentioned in \S\ref{ss:outline} is just the stratification into $G$-orbits. These nilpotent orbits $\cO\subset\cN$ are well known to be in bijection with the $G$-conjugacy classes of homomorphisms $\SL(2)\to G$: the bijection maps the $G$-conjugacy class of a homomorphism $\varphi:\SL(2)\to G$ to the $G$-orbit of $(d\varphi)([\begin{smallmatrix}0&1\\0&0\end{smallmatrix}])$. Moreover, a homomorphism $\varphi:\SL(2)\to G$ is determined up to $G$-conjugacy by its restriction to the diagonal subgroup $\Gg_m\subset\SL(2)$, and this restriction to $\Gg_m$ is $G$-conjugate to a unique dominant coweight $\lambda:\Gg_m\to T$, which must be an element of $\Lambda_1^+$. Hence every nilpotent orbit $\cO$ gives rise to an element $\lambda_\cO\in\Lambda^+_1$, often displayed visually as the \emph{weighted Dynkin diagram} of $\cO$.

In $\SL(2)[t,t^{-1}]$ we have the equation
\begin{equation} \label{eqn:sl2}
\exp([\begin{smallmatrix}0&t^{-1}\\0&0\end{smallmatrix}])=[\begin{smallmatrix}0&1\\-1&t\end{smallmatrix}]
[\begin{smallmatrix}t&0\\0&t^{-1}\end{smallmatrix}][\begin{smallmatrix}1&0\\t&1\end{smallmatrix}].
\end{equation}
Since $\exp\circ(d\varphi)=\varphi\circ\exp$, applying $\varphi:\SL(2)\to G$ to both sides of~\eqref{eqn:sl2} shows that $e(\cO)\subseteq \Gr_0^{\lambda_\cO}$, and hence
\begin{equation} \label{eqn:exp-inclusion}
e(\cO)\subseteq (\Gr_0^{\lambda_\cO})^\iota,\ \text{ for all nilpotent orbits $\cO$ of $G$.}
\end{equation} 
Usually the inclusion in~\eqref{eqn:exp-inclusion} is strict, but as a consequence of~\cite[Proposition 6.6]{ah}: 
\begin{equation} \label{eqn:exp-equality}
\text{If $\lambda_\cO$ is small, then }e(\cO)=(\Gr_0^{\lambda_\cO})^\iota.
\end{equation} 
Indeed, in this case, $\pi$ induces an isomorphism $(\Gr_0^{\lambda_\cO})^\iota\simto\cO$ whose inverse is the restriction of $e$ to $\cO$.

\begin{rmk}
Another phenomenon discovered in~\cite{ah} is that, when $G$ is almost simple and $\cO\subset\cN$ is a nilpotent orbit such that $\lambda_\cO$ is small and $(\Gr_0^{\lambda_\cO})^\iota$ is a proper subvariety of $\Gr_0^{\lambda_{\cO}}$, the quotient of the complement $\Gr_0^{\lambda_{\cO}}\smallsetminus(\Gr_0^{\lambda_\cO})^\iota$ by the involution $\iota$ is isomorphic via $\pi$ to a different nilpotent orbit $\cO'\subset\cN$, containing $\cO$ in its closure. The proof in~\cite{ah} was case-by-case; we hope that the moduli-space interpretation of $\iota$ given by Theorem~\ref{thm:normalizer} will assist in finding a uniform explanation. 
\end{rmk}

\subsection{An example}
\label{ss:sl2}

Suppose that $G=\SL(2)$. In this case we have $\Lambda_1^+=\Lambda^+=\{m\alpha\,|\,m\in\N\}$ where $\alpha$ is the positive coroot of $G$. Note that $m\alpha$ is small only when $m\in\{0,1\}$. The zero and nonzero nilpotent orbits have associated coweights $0$ and $\alpha$ respectively.

When $m=0$, $(\Gr_0^0)^\iota=\Gr_0^0=\{1\}$. When $m>0$, we have
\[
\begin{split}
\Gr_0^{m\alpha}&=\{1+x_1t^{-1}+x_2t^{-2}+\cdots+x_mt^{-m}\,|\,x_i\in\Mat(2),\, x_m\neq 0,\\
&\qquad\qquad\qquad\qquad\qquad\det(1+x_1t^{-1}+x_2t^{-2}+\cdots+x_mt^{-m})=1\},\\
(\Gr_0^{m\alpha})^\iota&=\{1+x_1t^{-1}+x_2t^{-2}+\cdots+x_mt^{-m}\,|\,x_i\in\Mat(2),\, x_m\neq 0,\\
&\qquad\qquad\qquad\qquad(1+x_1t^{-1}+x_2t^{-2}+\cdots+x_mt^{-m})\\
&\qquad\qquad\qquad\qquad\qquad\times(1-x_1t^{-1}+x_2t^{-2}-\cdots+(-1)^mx_mt^{-m})=1\}.
\end{split}
\]
The varieties $\Grbar_0^{m\alpha}=\bigsqcup_{m'\leq m} \Gr_0^{m'\alpha}$ and $(\Grbar_0^{m\alpha})^\iota=\bigsqcup_{m'\leq m}(\Gr_0^{m'\alpha})^\iota$ are described in the same way, but omitting the condition $x_m\neq 0$. In particular,
\begin{equation}
e(\cN)=\{1+xt^{-1}\,|\,x\in\Mat(2),\,x^2=0\}=(\Grbar_0^{\alpha})^\iota=\Grbar_0^{\alpha},
\end{equation}
exemplifying~\eqref{eqn:exp-equality}.  

\begin{lem} \label{lem:even-empty}
$(\Gr_0^{m\alpha})^\iota$ is empty when $m$ is even and positive.
\end{lem}

\begin{proof}
In the above description of $(\Gr_0^{m\alpha})^\iota$ for $m>0$, the equation forces the off-diagonal entries of $1+x_1t^{-1}+\cdots+x_mt^{-m}$ to be odd functions of $t^{-1}$; it also forces $x_m^2=0$. So if $m$ is even we have the contradictory requirements that $x_m\neq 0$, $x_m^2=0$, and $x_m$ is diagonal.
\end{proof}

We will describe $(\Gr_0^{m\alpha})^\iota$ for $m$ odd in \S\ref{ss:proof-gl2}, when we prove Theorem~\ref{thm:intro-gl2}.


\section{Moduli-space interpretations}
\label{sect:moduli}


To prove Theorem~\ref{thm:normalizer} (more specifically, part (2) of that statement), we need to revisit the moduli-space interpretation of $\Gr_0^\lambda$ given by Braverman and Finkelberg~\cite{bf}, which builds on their previous work with Gaitsgory~\cite{bfg}.

\subsection{Principal bundles on $\Pp^1$}
\label{ss:p1}

First recall the standard moduli-space interpretation of the affine Grassmannian $\Gr$ from~\cite[Section 5]{mv} (see also~\cite[Section 2D]{kwwy}). Namely, it is the moduli space $\Bun_G(\Pp^1;\Pp^1\smallsetminus 0)$ of principal $G$-bundles on $\Pp^1$ equipped with a chosen trivialization on $\Pp^1\smallsetminus 0$. 

Recall that any principal $G$-bundle $\cF$ on $\Pp^1$ must be trivial when restricted to $\Aa^1=\Pp^1\smallsetminus\infty$ and when restricted to $\Pp^1\smallsetminus 0$. So $\cF$ can be constructed by gluing together trivial principal $G$-bundles on $\Aa^1$ and on $\Pp^1\smallsetminus 0$ using some transition function on $\Aa^1\smallsetminus 0$, i.e.\ an element of $G[t,t^{-1}]$. We obtain an isomorphic bundle if we left multiply by an element of $G[t^{-1}]$ (changing the trivialization on $\Pp^1\smallsetminus 0$) or right multiply by an element of $G[t]$ (changing the trivialization on $\Aa^1$). So the moduli space of principal $G$-bundles on $\Pp^1$ can be expressed as a quotient stack:
\begin{equation} 
\Bun_G(\Pp^1)\cong G[t^{-1}]\backslash G[t,t^{-1}]/G[t].
\end{equation}
Note that the action of $G[t^{-1}]\times G[t]$ on $G[t,t^{-1}]$ is not free, which corresponds to the fact that principal $G$-bundles on $\Pp^1$ have nontrivial automorphism groups. Indeed, even a trivial principal $G$-bundle on $\Pp^1$ has automorphism group $G$; this trivial bundle is parametrized by the open substack $G[t^{-1}]\backslash G[t^{-1}]_1G[t]/G[t]\cong G\backslash\textrm{pt}$, corresponding to the `big opposite Bruhat cell' $G[t^{-1}]_1G[t]=G[t^{-1}]G[t]$ of $G[t,t^{-1}]$.

If we include the trivialization on $\Pp^1\smallsetminus 0$ as part of the data, then only the (free) action of $G[t]$ remains, so we get
\begin{equation} 
\Bun_G(\Pp^1;\Pp^1\smallsetminus 0)\cong G[t,t^{-1}]/G[t]=\Gr,
\end{equation}
now not merely a stack but an ind-variety. In this moduli space, the open subspace parametrizing trivial bundles is $G[t^{-1}]_1G[t]/G[t]=\Grom$.

Note that if instead we included only the trivialization at $\infty$ as part of the data, we would get the moduli stack
\begin{equation} \label{eqn:stack}
\Bun_G(\Pp^1;\infty)\cong G[t^{-1}]_1\backslash G[t,t^{-1}]/G[t],
\end{equation}
in which the open substack parametrizing trivial bundles is a point.

\subsection{Principal bundles on $\Pp^1\times\Pp^1$}

The Braverman--Finkelberg result~\cite[Theorem 5.2]{bf} takes place one dimension up from the above picture, on a rational surface (it is a by-product of their pursuit of the `double affine Grassmannian'). 

The setting is the moduli space $\Bun_G(\Aa^2)$, introduced in~\cite{atiyah,bfg,donaldson}. In one definition,
\begin{equation} \label{eqn:old-def}
\Bun_G(\Aa^2):=\Bun_G(\Pp^2;\Pp^2\smallsetminus\Aa^2)
\end{equation}
is the moduli space of principal $G$-bundles on $\Pp^2$ equipped with a chosen trivialization on the complement of $\Aa^2$. As such, $\Bun_G(\Aa^2)$ clearly carries an action of $G\times \GL(2)$, where $G$ acts by changing the trivialization on $\Pp^2\smallsetminus\Aa^2$, and $\GL(2)$ acts on the base space $\Pp^2$, preserving $\Aa^2$. However, to define the map $\Psi$ referred to in Theorem~\ref{thm:normalizer}, Braverman and Finkelberg use the alternative definition
\begin{equation} \label{eqn:new-def}
\Bun_G(\Aa^2):=\Bun_G(\Pp^1\times\Pp^1;(\Pp^1\times\Pp^1)\smallsetminus\Aa^2).
\end{equation}
In this definition the action of $\GL(2)$ is obscured, but the action of the subgroup $N_{\GL(2)}(\Gg_m)$ appearing in Theorem~\ref{thm:normalizer} remains. So we will use~\eqref{eqn:new-def} until we have proved that theorem, and then explain (in \S\ref{ss:p2}) how to switch to~\eqref{eqn:old-def}.

\begin{rmk} \label{rmk:chern}
Braverman and Finkelberg rarely refer to the whole moduli space $\Bun_G(\Aa^2)$, because $\Bun_G(\Aa^2)$ can easily be decomposed into its connected components by considering the second Chern class of a bundle. In view of~\cite[Theorem 5.2 part (1)]{bf}, this decomposition will be superseded once we impose the $\Gg_m$-equivariance, so we do not need to recall it here.
\end{rmk}

Set $D_\infty=(\Pp^1\times\Pp^1)\smallsetminus\Aa^2$. Then a (closed) point of $\Bun_G(\Aa^2)$ is an isomorphism class of pairs $(\cF,\Phi)$ where $\cF$ is a principal $G$-bundle on $\Pp^1\times\Pp^1$ and $\Phi:\cF|_{D_\infty}\simto G\times D_\infty$ is an isomorphism of principal $G$-bundles on $D_\infty$, i.e.\ a trivialization of $\cF$ on $D_\infty$. In a convenient abuse of notation, we will write $(\cF,\Phi)\in\Bun_G(\Aa^2)$ to mean that $(\cF,\Phi)$ is such a pair.

We now present $\Bun_G(\Aa^2)$ as a quotient, in the same spirit as~\eqref{eqn:stack}. We use the obvious open covering of $\Pp^1\times\Pp^1$ by the following four copies of $\Aa^2$:
\[
\begin{split}
U_{00}&:=\Aa^2=\Aa^1\times\Aa^1,\\
U_{01}&:= \Aa^1\times(\Pp^1\smallsetminus 0),\\
U_{1 0}&:= (\Pp^1\smallsetminus 0)\times\Aa^1,\\  
U_{1 1}&:=(\Pp^1\smallsetminus 0)\times(\Pp^1\smallsetminus 0).
\end{split}
\]
Let $t$ and $u$ be coordinates on the first and second $\Pp^1$ factors, respectively. Then the morphisms $U_{00}\to G$ form the ind-group $G[t,u]$, and similarly the other open sets give rise to $G[t,u^{-1}]$, $G[t^{-1},u]$ and $G[t^{-1},u^{-1}]$ respectively. We will use the notation $(h_{00},h_{01},h_{1 0},h_{11})$ for an element of the product
\begin{equation} \label{eqn:group}
G[t,u]\times G[t,u^{-1}]\times G[t^{-1},u]\times G[t^{-1},u^{-1}].
\end{equation}

A principal $G$-bundle $\cF$ on $\Pp^1\times\Pp^1$ must be trivial on each of the above four open sets, so it can be constructed by gluing together trivial principal $G$-bundles on these open sets using transition functions on the overlaps. The transition functions constitute a quadruple $(g_{01}^{00},g_{10}^{00},g_{11}^{01},g_{11}^{10})$, where
\begin{equation} \label{eqn:containments}
\begin{split}
g_{01}^{00}&\in G[t,u,u^{-1}],\\
g_{10}^{00}&\in G[t,t^{-1},u],\\
g_{11}^{01}&\in G[t,t^{-1},u^{-1}],\\
g_{11}^{10}&\in G[t^{-1},u,u^{-1}],
\end{split}
\end{equation}
and these must satisfy the equation
\begin{equation} \label{eqn:cocycle}
g_{11}^{01}g_{01}^{00}=g_{11}^{10}g_{10}^{00}\qquad \text{in $G[t,t^{-1},u,u^{-1}]$.}
\end{equation}
Here, $g_{01}^{00}$ gives the transition function on $U_{00}\cap U_{01}$, and so forth; the two equal sides of~\eqref{eqn:cocycle} give the transition function on $U_{00}\cap U_{11}$, and the transition function on $U_{01}\cap U_{10}$ is given by the two equal sides of the equivalent equation
\begin{equation} \label{eqn:cocycle2}
(g_{11}^{01})^{-1}g_{11}^{10}=g_{01}^{00}(g_{10}^{00})^{-1}\qquad \text{in $G[t,t^{-1},u,u^{-1}]$.}
\end{equation}

We conclude that the moduli space $\Bun_G(\Pp^1\times\Pp^1)$, as a stack, is the quotient of the space of quadruples $(g_{01}^{00},g_{10}^{00},g_{11}^{01},g_{11}^{10})$ satisfying~\eqref{eqn:containments} and~\eqref{eqn:cocycle} by the group~\eqref{eqn:group}, where the action is given by:
\begin{equation} \label{eqn:action}
\begin{split}
(h_{00},h_{01},h_{1 0},h_{11})&\cdot(g_{01}^{00},g_{10}^{00},g_{11}^{01},g_{11}^{10})\\
&=(h_{01}g_{01}^{00}h_{00}^{-1},h_{10}g_{10}^{00}h_{00}^{-1},h_{11}g_{11}^{01}h_{01}^{-1},h_{11}g_{11}^{10}h_{10}^{-1}).
\end{split}
\end{equation} 

To obtain $\Bun_G(\Aa^2)$, we must modify this description to incorporate the trivialization $\Phi$ of $\cF$ on $D_\infty=(U_{01}\cup U_{10}\cup U_{11})\smallsetminus U_{00}$. First, we must impose two additional conditions on the quadruple $(g_{01}^{00},g_{10}^{00},g_{11}^{01},g_{11}^{10})$, namely that
\begin{equation} \label{eqn:rigidity}
g_{11}^{01}\in G[t,t^{-1},u^{-1}]_1^u\text{ and }g_{11}^{10}\in G[t^{-1},u,u^{-1}]_1^t,
\end{equation}
where $G[t,t^{-1},u^{-1}]_1^u$ means $\ker(G[t,t^{-1},u^{-1}]\to G[t,t^{-1}])$ and $G[t^{-1},u,u^{-1}]_1^t$ is defined analogously. The first condition in~\eqref{eqn:rigidity} ensures that the given trivializations on $U_{01}\smallsetminus U_{00}$ and $U_{11}\smallsetminus U_{00}$ match up along their overlap, and the second condition has a similar interpretation. Let $Z$ denote the space of quadruples $(g_{01}^{00},g_{10}^{00},g_{11}^{01},g_{11}^{10})$ satisfying~\eqref{eqn:containments}, \eqref{eqn:cocycle} and~\eqref{eqn:rigidity}. Any element of $Z$ gives rise to a pair $(\cF,\Phi)$ of the sort we are trying to parametrize.

Second, to avoid changing the isomorphism class of $(\cF,\Phi)$, we must take a quotient not by the full group~\eqref{eqn:group} but by its subgroup
\begin{equation} \label{eqn:subgroup}
H:=G[t,u]\times G[t,u^{-1}]_1^u\times G[t^{-1},u]_1^t\times G[t^{-1},u^{-1}]_1^{t,u},
\end{equation}
where $G[t,u^{-1}]_1^u$ means $\ker(G[t,u^{-1}]\to G[t])$, $G[t^{-1},u]_1^t$ is defined analogously, and $G[t^{-1},u^{-1}]_1^{t,u}$ is the intersection of $G[t^{-1},u^{-1}]_1^{t}$ and $G[t^{-1},u^{-1}]_1^{u}$. To sum up,
\begin{equation} \label{eqn:p1p1presentation}
\Bun_G(\Aa^2)=\Bun_G(\Pp^1\times\Pp^1;D_\infty)\cong Z/H
\end{equation}
as a quotient stack, where the action of $H$ on $Z$ is given by~\eqref{eqn:action} as before.

In this description, the action of $G$ on $\Bun_G(\Aa^2)$ results from an action of $G$ on $Z$, obtained by regarding $G$ as the subgroup of the group~\eqref{eqn:group} responsible for a uniform change of the trivialization $\Phi$: namely, for $g\in G$ and $(g_{01}^{00},g_{10}^{00},g_{11}^{01},g_{11}^{10})\in Z$, we have
\begin{equation} \label{eqn:G-action}
g\cdot(g_{01}^{00},g_{10}^{00},g_{11}^{01},g_{11}^{10})=(gg_{01}^{00},gg_{10}^{00},gg_{11}^{01}g^{-1},gg_{11}^{10}g^{-1}).
\end{equation} 
The action of the diagonal subgroup of $\GL(2)$ on $\Bun_G(\Aa^2)$ results from the following action on $Z$:
\begin{equation} \label{eqn:diag-action}
[\begin{smallmatrix}\alpha&0\\0&\beta\end{smallmatrix}]\cdot(g_{01}^{00},g_{10}^{00},g_{11}^{01},g_{11}^{10})
=(g_{01}^{00},g_{10}^{00},g_{11}^{01},g_{11}^{10})\left|_{\substack{t\mapsto \alpha^{-1}t\\u\mapsto \beta^{-1}u}}\right.\, ,
\end{equation}
and the action of $[\begin{smallmatrix}0&1\\1&0\end{smallmatrix}]$ on $\Bun_G(\Aa^2)$ results from the following action on $Z$:
\begin{equation} \label{eqn:switch-action}
[\begin{smallmatrix}0&1\\1&0\end{smallmatrix}]\cdot(g_{01}^{00},g_{10}^{00},g_{11}^{01},g_{11}^{10})
=(g_{10}^{00},g_{01}^{00},g_{11}^{10},g_{11}^{01})\left|_{\substack{t\mapsto u\\u\mapsto t}}\right.\,.
\end{equation}

The following fact is implicit in the proof of~\cite[Proposition 3.4]{bfg}.

\begin{prop} \label{prop:factorization}
If $(g_{01}^{00},g_{10}^{00},g_{11}^{01},g_{11}^{10})\in Z$, then, regarding $g_{11}^{01}$ and $g_{11}^{10}$ as elements of $G[t,t^{-1}][[u^{-1}]]_1^u$ and $G[[t^{-1}]][u,u^{-1}]_1^t$ respectively, we can uniquely write them in the form
\[ g_{11}^{01}=(g_{11}^{01})'(g_{11}^{01})''\;\text{ and }\;g_{11}^{10}=(g_{11}^{10})'(g_{11}^{10})'' \] 
where
\[
\begin{split}
&(g_{11}^{01})'\in G[t^{-1}][[u^{-1}]]_1^{t,u},\ (g_{11}^{01})''\in G[t][[u^{-1}]]_1^u,\\
&(g_{11}^{10})'\in G[[t^{-1}]][u^{-1}]_1^{t,u},\ (g_{11}^{10})''\in G[[t^{-1}]][u]_1^t.
\end{split}
\] 
\end{prop} 

\begin{proof}
By symmetry we need only prove the claim concerning $g_{11}^{01}$. The uniqueness in the claimed expression $g_{11}^{01}=(g_{11}^{01})'(g_{11}^{01})''$ is clear, because 
\[ G[t^{-1}][[u^{-1}]]_1^{t,u}\cap G[t][[u^{-1}]]_1^u=\{1\}. \]
Recall from the previous subsection that when a principal $G$-bundle on $\Pp^1$ is represented by a transition function $g\in G[t,t^{-1}]$, the bundle is trivial exactly when $g\in G[t^{-1}]_1 G[t]$. So if $(g_{01}^{00},g_{10}^{00},g_{11}^{01},g_{11}^{10})\in Z$ gives rise to $(\cF,\Phi)\in\Bun_G(\Aa^2)$, the claim that $g_{11}^{01}\in G[t^{-1}][[u^{-1}]]_1^{t,u}G[t][[u^{-1}]]_1^u$ is equivalent to saying that $\cF$ is trivial when restricted to $\Pp^1\times \cD$, where $\cD$ is a formal neighbourhood of $\infty$ in $\Pp^1$; this is the form in which the claim appears in the proof of~\cite[Proposition 3.4]{bfg}. The reason is that $\cF$ is assumed to be trivial when restricted to $\Pp^1\times\{\infty\}$, and the set of points $u_0\in\Pp^1$ such that $\cF$ is nontrivial on $\Pp^1\times\{u_0\}$ must be Zariski-closed, hence finite; these finitely many lines $\Pp^1\times\{u_0\}$ are known as the `jumping lines' of $\cF$. 
\end{proof}

As a consequence we have the following result, mentioned in~\cite[Section 3.5]{bfg} and~\cite[Section 4.4]{bf}, which is the reason that $\Bun_G(\Aa^2)\cong Z/H$ is not merely a stack but an ind-variety.

\begin{prop} \label{prop:free}
The action of $H$ on $Z$ is free. In other words, a pair $(\cF,\Phi)\in\Bun_G(\Aa^2)$ has no nontrivial automorphisms.
\end{prop} 

\begin{proof}
Suppose that we have
\begin{equation} \label{eqn:fix-eqn}
(h_{00},h_{01},h_{10},h_{11})\cdot(g_{01}^{00},g_{10}^{00},g_{11}^{01},g_{11}^{10})
=(g_{01}^{00},g_{10}^{00},g_{11}^{01},g_{11}^{10})
\end{equation}
for some $(h_{00},h_{01},h_{1 0},h_{11})\in H$ and $(g_{01}^{00},g_{10}^{00},g_{11}^{01},g_{11}^{10})\in Z$.
Taking the third component of both sides of~\eqref{eqn:fix-eqn} gives 
\begin{equation} \label{eqn:fix-eqn3}
h_{11}g_{11}^{01}h_{01}^{-1}=g_{11}^{01}.
\end{equation}
Writing $g_{11}^{01}=(g_{11}^{01})'(g_{11}^{01})''$ as in Proposition~\ref{prop:factorization} and using the uniqueness in that expression, we obtain $h_{11}(g_{11}^{01})'=(g_{11}^{01})'$ and $(g_{11}^{01})''h_{01}^{-1}=(g_{11}^{01})''$, implying $h_{11}=1$ and $h_{01}=1$. Then the other components of~\eqref{eqn:fix-eqn} immediately imply $h_{00}=1$ and $h_{10}=1$ also.
\end{proof} 

As observed in~\cite[Section 4.4]{bf}, it follows from Proposition~\ref{prop:free} that for any group $\Gamma$ acting on $\Pp^1\times\Pp^1$ in such a way as to preserve $\Aa^2$, the moduli space $\Bun_G(\Aa^2/\Gamma)$ of $\Gamma$-equivariant pairs $(\cF,\Phi)$ as above can be identified with the fixed-point set $\Bun_G(\Aa^2)^\Gamma$ of the $\Gamma$-action on $\Bun_G(\Aa^2)$. That is, for a pair $(\cF,\Phi)\in\Bun_G(\Aa^2)$, $\Gamma$-equivariance can be treated as a property rather than as extra structure.

\subsection{$\Gg_m$-equivariant bundles}
\label{ss:gm-equiv}

Following~\cite[Section 5]{bf}, we now describe the moduli space $\Bun_G(\Aa^2/\Gg_m)=\Bun_G(\Aa^2)^{\Gg_m}$ where $\Gg_m$ is the diagonal subgroup of $\SL(2)$. Note that since $\Aa^2/\Gg_m$ is one-dimensional, we should not be surprised to find subvarieties of the ordinary one-variable affine Grassmannian reappearing in the description.

Parts of the next two results are implicit in the proof of~\cite[Theorem 5.2]{bf}.

\begin{prop} \label{prop:psi-prelim}
Suppose that $(g_{01}^{00},g_{10}^{00},g_{11}^{01},g_{11}^{10})\in Z$ is such that its $H$-orbit is fixed by the $\Gg_m$-action. Then:
\begin{enumerate} 
\item We can uniquely write $g_{11}^{01}$ and $g_{11}^{10}$ in the form
\[
g_{11}^{01}=(g_{11}^{01})'(g_{11}^{01})''\;\text{ and }\;g_{11}^{10}=(g_{11}^{10})'(g_{11}^{10})'' \] 
where
\[
\begin{split}
&(g_{11}^{01})'\in G[t^{-1},u^{-1}]_1^{t,u},\ (g_{11}^{01})''\in G[t,u^{-1}]_1^u,\\
&(g_{11}^{10})'\in G[t^{-1},u^{-1}]_1^{t,u},\ (g_{11}^{10})''\in G[t^{-1},u]_1^t.
\end{split}
\]
\item With the above notation, we have
\[
((g_{11}^{01})')^{-1}(g_{11}^{10})'=\gamma|_{t\mapsto tu}
\] 
for some, clearly unique, $\gamma\in G[t^{-1}]_1=\Grom$.
\end{enumerate}
\end{prop}

\begin{proof}
The proof of (1) is similar to that of Proposition~\ref{prop:factorization}. By symmetry, we need only prove the claim concerning $g_{11}^{01}$, and the uniqueness of the claimed expression is obvious. If $(g_{01}^{00},g_{10}^{00},g_{11}^{01},g_{11}^{10})$ gives rise to $(\cF,\Phi)\in\Bun_G(\Aa^2)$, the claim that $g_{11}^{01}\in G[t^{-1},u^{-1}]_1^{t,u}G[t,u^{-1}]_1^u$ is equivalent to saying that $\cF$ is trivial when restricted to $\Pp^1\times(\Pp^1\smallsetminus 0)$. The reason, as mentioned in the proof of~\cite[Theorem 5.2]{bf}, is that since $\cF$ is $\Gg_m$-equivariant, its set of jumping lines $\Pp^1\times\{u_0\}$ must be $\Gg_m$-stable, which means that the only possible jumping line is $\Pp^1\times\{0\}$.

Now we prove (2). The assumption that the $H$-orbit of $(g_{01}^{00},g_{10}^{00},g_{11}^{01},g_{11}^{10})\in Z$ is fixed by the $\Gg_m$-action is equivalent to saying that for any $z\in\Gg_m$, there is some $(h_{00}^z,h_{01}^z,h_{10}^z,h_{11}^z)\in H$ such that
\begin{equation} \label{eqn:z-fixed}
\begin{split}
&(g_{01}^{00},g_{10}^{00},g_{11}^{01},g_{11}^{10})\left|_{\substack{t\mapsto z^{-1}t\\u\mapsto zu\phantom{{}^{-1}}}}\right.\\
&=(h_{01}^z g_{01}^{00}(h_{00}^z)^{-1},h_{10}^z g_{10}^{00}(h_{00}^z)^{-1},h_{11}^z g_{11}^{01}(h_{01}^z)^{-1},h_{11}^z g_{11}^{10}(h_{10}^z)^{-1}).
\end{split}
\end{equation}
Note that by Proposition~\ref{prop:free}, $(h_{00}^z,h_{01}^z,h_{10}^z,h_{11}^z)$ is uniquely determined by $z$.

Taking the third component of both sides of~\eqref{eqn:z-fixed} gives
\begin{equation}
g_{11}^{01}\left|_{\substack{t\mapsto z^{-1}t\\u\mapsto zu\phantom{{}^{-1}}}}\right.=h_{11}^z g_{11}^{01}(h_{01}^z)^{-1}.
\end{equation}
Writing $g_{11}^{01}=(g_{11}^{01})'(g_{11}^{01})''$ as in (1) and using the uniqueness in that expression, we obtain
\begin{equation} \label{eqn:split1}
(g_{11}^{01})'\left|_{\substack{t\mapsto z^{-1}t\\u\mapsto zu\phantom{{}^{-1}}}}\right.=h_{11}^z (g_{11}^{01})'.
\end{equation}
By symmetry, we also have
\begin{equation} \label{eqn:split2}
(g_{11}^{10})'\left|_{\substack{t\mapsto z^{-1}t\\u\mapsto zu\phantom{{}^{-1}}}}\right.=h_{11}^z (g_{11}^{10})'.
\end{equation}
Combining~\eqref{eqn:split1} and~\eqref{eqn:split2}, we see that $((g_{11}^{01})')^{-1}(g_{11}^{10})'\in G[t^{-1},u^{-1}]_1^{t,u}$ is preserved by the substitution $t\mapsto z^{-1}t$, $u\mapsto zu$ for all $z\in\Gg_m$, proving part (2).
\end{proof}

\begin{prop} \label{prop:psi}
There is a $G$-equivariant bijection
\[ \Psi:\Bun_G(\Aa^2/\Gg_m)\to\Grom \]
which sends the $H$-orbit of $(g_{01}^{00},g_{10}^{00},g_{11}^{01},g_{11}^{10})\in Z$ to $\gamma\in\Grom$, defined as in Proposition~\ref{prop:psi-prelim}\textup{(}2\textup{)}. The inverse is described as follows. If $\gamma\in\Gr_0^\lambda$ for $\lambda\in\Lambda^+$, i.e.\
\[
\gamma=q_1 t^\lambda q_2\;\text{ for some }\,q_1,q_2\in G[t],
\]
then $\Psi^{-1}(\gamma)$ is the $H$-orbit of the quadruple $(\tilde{g}_{01}^{00},\tilde{g}_{10}^{00},\tilde{g}_{11}^{01},\tilde{g}_{11}^{10})\in Z$ defined by
\[
\begin{split}
\tilde{g}_{01}^{00}&=(q_1|_{t\mapsto tu})\,u^\lambda,\\
\tilde{g}_{10}^{00}&=(q_2^{-1}|_{t\mapsto tu})\,t^{-\lambda},\\
\tilde{g}_{11}^{01}&=1,\\
\tilde{g}_{11}^{10}&=\gamma|_{t\mapsto tu}.
\end{split}
\]
Here $u^\lambda\in T[u,u^{-1}]$ means the same as $t^{\lambda}$, but written with the variable $u$.
\end{prop} 

\begin{ex} \label{ex:sl2-special}
Take $G=\SL(2)$ and let $\gamma=[\begin{smallmatrix}1&t^{-1}\\0&1\end{smallmatrix}]\in\Gr_0^\alpha$. By~\eqref{eqn:sl2}, we can choose $q_1=[\begin{smallmatrix}0&1\\-1&t\end{smallmatrix}]$ and $q_2=[\begin{smallmatrix}1&0\\t&1\end{smallmatrix}]$. Then the special representative of the $H$-orbit $\Psi^{-1}(\gamma)$ specified in Proposition~\ref{prop:psi} is $([\begin{smallmatrix}0&u^{-1}\\-u&t\end{smallmatrix}],[\begin{smallmatrix}t^{-1}&0\\-u&t\end{smallmatrix}],[\begin{smallmatrix}1&0\\0&1\end{smallmatrix}],[\begin{smallmatrix}1&t^{-1}u^{-1}\\0&1\end{smallmatrix}])\in Z$.
\end{ex}

\begin{proof}
We first show that $\Psi$ is well defined. Suppose that $(g_{01}^{00},g_{10}^{00},g_{11}^{01},g_{11}^{10})\in Z$ is such that its $H$-orbit is fixed by the $\Gg_m$-action, and define $(g_{11}^{01})'$ and $(g_{11}^{10})'$ as in Proposition~\ref{prop:psi-prelim}. If we act on $(g_{01}^{00},g_{10}^{00},g_{11}^{01},g_{11}^{10})$ by $(h_{00},h_{01},h_{10},h_{11})\in H$, then $(g_{11}^{01})'$ is replaced by $h_{11}(g_{11}^{01})'$ and $(g_{11}^{10})'$ by $h_{11}(g_{11}^{10})'$, so $((g_{11}^{01})')^{-1}(g_{11}^{10})'$ is left unchanged. Thus the element $\gamma\in G[t^{-1}]_1$ defined in Proposition~\ref{prop:psi-prelim}(2) is indeed an invariant of the $H$-orbit of $(g_{01}^{00},g_{10}^{00},g_{11}^{01},g_{11}^{10})$. The fact that $\Psi$ is $G$-equivariant is equally easy, using the formula for the $G$-action on $Z$ given in~\eqref{eqn:G-action}.

To show the surjectivity of $\Psi$, it suffices to show that $(\tilde{g}_{01}^{00},\tilde{g}_{10}^{00},\tilde{g}_{11}^{01},\tilde{g}_{11}^{10})$ as defined in the statement does indeed have the property that its $H$-orbit is fixed by the $\Gg_m$-action. This follows from the obvious equation
\begin{equation} \label{eqn:z-fixed-special}
(\tilde{g}_{01}^{00},\tilde{g}_{10}^{00},\tilde{g}_{11}^{01},\tilde{g}_{11}^{10})\left|_{\substack{t\mapsto z^{-1}t\\u\mapsto zu\phantom{{}^{-1}}}}\right.\\
=(\tilde{g}_{01}^{00}\lambda(z),\tilde{g}_{10}^{00}\lambda(z),\tilde{g}_{11}^{01},\tilde{g}_{11}^{10})
\end{equation}
for all $z\in\Gg_m$. 

To prove the injectivity of $\Psi$, we must show that if $\Psi$ sends the $H$-orbit of $(g_{01}^{00},g_{10}^{00},g_{11}^{01},g_{11}^{10})\in Z$ to $\gamma\in\Gr_0^\lambda$, then that $H$-orbit contains the above quadruple $(\tilde{g}_{01}^{00},\tilde{g}_{10}^{00},\tilde{g}_{11}^{01},\tilde{g}_{11}^{10})$. Now~\eqref{eqn:cocycle2} implies that
\[
(g_{11}^{01})''g_{01}^{00}(g_{10}^{00})^{-1}((g_{11}^{10})'')^{-1}=((g_{11}^{01})')^{-1}(g_{11}^{10})'=(q_1|_{t\mapsto tu})t^\lambda u^\lambda (q_2|_{t\mapsto tu}),
\]
which rearranges to
\begin{equation} \label{eqn:key}
u^{-\lambda}(q_1^{-1}|_{t\mapsto tu})(g_{11}^{01})''g_{01}^{00}=t^\lambda (q_2|_{t\mapsto tu})(g_{11}^{10})''g_{10}^{00}.
\end{equation}
The left-hand side of~\eqref{eqn:key} clearly belongs to $G[t,u,u^{-1}]$ and the right-hand side clearly belongs to $G[t,t^{-1},u]$, so their common value is some element $h_{00}$ of $G[t,u]$. We then have
\begin{equation}
(h_{00},(g_{11}^{01})'',(g_{11}^{10})'',((g_{11}^{01})')^{-1})\cdot(g_{01}^{00},g_{10}^{00},g_{11}^{01},g_{11}^{10})=(\tilde{g}_{01}^{00},\tilde{g}_{10}^{00},\tilde{g}_{11}^{01},\tilde{g}_{11}^{10}),
\end{equation}
as required.
\end{proof}

\subsection{Proof of Theorem~\ref{thm:normalizer}}
\label{ss:proof}

We have used the same letter $\Psi$ for the bijection defined in Proposition~\ref{prop:psi} as we used in the introduction for the bijection defined by Braverman and Finkelberg. Of course, this is to be justified by the next result, which says that the two bijections are the same. 

Recall from the introduction that for any $\lambda\in\Lambda^+$, $\Bun_G^\lambda(\Aa^2/\Gg_m)$ denotes the closed subvariety of $\Bun_G(\Aa^2/\Gg_m)$ parametrizing $\Gg_m$-equivariant pairs $(\cF,\Phi)$ for which the induced action of $\Gg_m$ on the fibre of $\cF$ at $(0,0)\in\Aa^2$ gives rise to a homomorphism $\Gg_m\to G$ in the $G$-conjugacy class of $\lambda$. As observed in~\cite{bf}, $\Bun_G(\Aa^2/\Gg_m)$ decomposes as a disconnected union $\coprod_{\lambda\in\Lambda^+}\Bun_G^\lambda(\Aa^2/\Gg_m)$.

\begin{prop} \label{prop:identification}
The bijection $\Psi:\Bun_G(\Aa^2/\Gg_m)\to\Grom$ defined in Proposition~\ref{prop:psi} restricts to an isomorphism of varieties
\[
\Bun_G^\lambda(\Aa^2/\Gg_m)\simto\Gr_0^\lambda\ \text{ for each }\lambda\in\Lambda^+,
\]
which is the same as the isomorphism defined by Braverman and Finkelberg in the $\mu=0$ special case of ~\cite[Theorem 5.2(2)]{bf}.
\end{prop}

\begin{proof}
Let $\gamma\in\Grom$, let $\lambda\in\Lambda^+$ be such that $\gamma\in\Gr_0^\lambda$, and let $(\tilde{g}_{01}^{00},\tilde{g}_{10}^{00},\tilde{g}_{11}^{01},\tilde{g}_{11}^{10})\in Z$ be the special representative of $\Psi^{-1}(\gamma)$ defined in Proposition~\ref{prop:psi}. This quadruple $(\tilde{g}_{01}^{00},\tilde{g}_{10}^{00},\tilde{g}_{11}^{01},\tilde{g}_{11}^{10})$ defines a pair $(\cF,\Phi)\in\Bun_G(\Aa^2)$ as usual. The fact that $\tilde{g}_{11}^{01}=1$ means that the restriction of the trivialization $\Phi$ to $\Pp^1\times\{\infty\}$ extends to $\Pp^1\times(\Pp^1\smallsetminus 0)$.  

The first component of~\eqref{eqn:z-fixed-special} exactly says that the action of $\Gg_m$ on the fibre of $\cF$ at $(0,0)\in\Aa^2$ gives rise to the homomorphism $\lambda:\Gg_m\to G$. We conclude that $\Psi^{-1}(\Gr_0^\lambda)$ is indeed $\Bun_G^\lambda(\Aa^2/\Gg_m)$. 

It is not hard to show directly that $\Psi$ restricts to an isomorphism of varieties $\Psi^{-1}(\Gr_0^\lambda)\simto\Gr_0^\lambda$, but we will omit this since it follows anyway from the identification with the known isomorphism. 

Braverman and Finkelberg describe their isomorphism $\Bun_G^\lambda(\Aa^2/\Gg_m)\simto\Gr_0^\lambda$ in the proof of~\cite[Theorem 5.2]{bf}, which incorporates the proof of~\cite[Lemma 5.3]{bf}. They use a principle established in~\cite[Proposition 3.4]{bfg}, which states, in part, that if the restriction of $\Phi$ to $\Pp^1\times\{\infty\}$ extends to $\Pp^1\times(\Pp^1\smallsetminus 0)$, then the isomorphism class of $(\cF,\Phi)$ is determined by the morphism $f_{(\cF,\Phi)}:\Pp^1\to\Bun_G(\Pp^1;\Pp^1\smallsetminus 0)=\Gr$ which maps each $t_0\in\Pp^1$ to the isomorphism class of the restriction of $(\cF,\Phi)$ to $\{t_0\}\times\Pp^1$. Then the isomorphism $\Bun_G^\lambda(\Aa^2/\Gg_m)\simto\Gr_0^\lambda$ sends $(\cF,\Phi)$ to $f_{(\cF,\Phi)}(1)$. In the case of the pair $(\cF,\Phi)$ determined by $(\tilde{g}_{01}^{00},\tilde{g}_{10}^{00},\tilde{g}_{11}^{01},\tilde{g}_{11}^{10})$ as above, the fact that $\tilde{g}_{11}^{10}=\gamma|_{t\mapsto tu}$ means that $f_{(\cF,\Phi)}(t_0)=\gamma|_{t\mapsto t_0 t}$ for all $t_0\in\Aa^1\smallsetminus 0$, so the Braverman--Finkelberg isomorphism sends $(\cF,\Phi)$ to $\gamma$ as claimed.
\end{proof}

With our explicit definition of $\Psi$, the proof of Theorem~\ref{thm:normalizer} is trivial:
\begin{proof}
Suppose that $(g_{01}^{00},g_{10}^{00},g_{11}^{01},g_{11}^{10})\in Z$ is such that its $H$-orbit is fixed by the $\Gg_m$-action, and define $(g_{11}^{01})'$, $(g_{11}^{10})'$, $\gamma$ as in Proposition~\ref{prop:psi-prelim}. In view of~\eqref{eqn:diag-action}, acting by $[\begin{smallmatrix}\alpha&0\\0&\beta\end{smallmatrix}]$ changes $(g_{11}^{01})'$ and $(g_{11}^{10})'$ by the substitutions $t\mapsto \alpha^{-1}t$, $u\mapsto\beta^{-1}u$, so it changes $\gamma$ by the substitution $t\mapsto \alpha^{-1}\beta^{-1} t$, as claimed. In view of~\eqref{eqn:switch-action}, acting by $[\begin{smallmatrix}0&1\\1&0\end{smallmatrix}]$ swaps $(g_{11}^{01})'$ and $(g_{11}^{10})'$ and interchanges $t$ and $u$, so it replaces $\gamma$ by $\gamma^{-1}$, as claimed.\end{proof}

\begin{rmk} \label{rmk:uhlenbeck}
Another, equally important, part of~\cite[Theorem 5.2]{bf} is the statement that the isomorphism $\Bun_G^\lambda(\Aa^2/\Gg_m)\simto\Gr_0^\lambda$ extends to an isomorphism between natural `closures' of these varieties, namely the Uhlenbeck moduli space $\cU_G^\lambda(\Aa^2/\Gg_m)$ and the variety $\Grbar_0^\lambda$. (Really, as mentioned in Remark~\ref{rmk:slice}, Braverman and Finkelberg consider the more general situation of the transverse slice $\Grbar_\mu^\lambda$ where $\mu$ is not necessarily $0$.) Since an automorphism of a variety is determined by its restriction to a dense subset, it follows immediately from Theorem~\ref{thm:normalizer} that the action of $N_{\GL(2)}(\Gg_m)$ on $\cU_G^\lambda(\Aa^2/\Gg_m)\cup\cU_G^{-w_0\lambda}(\Aa^2/\Gg_m)$ corresponds to the action on $\Grbar_0^\lambda\cup\Grbar_0^{-w_0\lambda}$ by the same formulas as in Theorem~\ref{thm:normalizer}. 
\end{rmk}   

\subsection{$N$-equivariant bundles}
\label{ss:n-equiv}

We now turn to the moduli space $\Bun_G(\Aa^2/N)$. As explained after Proposition~\ref{prop:free}, $\Bun_G(\Aa^2/N)$ can be identified with the fixed-point set in $\Bun_G(\Aa^2)$ of the action of $N$, and since $N=\langle\Gg_m,[\begin{smallmatrix}0&1\\-1&0\end{smallmatrix}]\rangle$, this is the same as the fixed-point set in $\Bun_G(\Aa^2/\Gg_m)$ of the action of $[\begin{smallmatrix}0&1\\-1&0\end{smallmatrix}]$. So most of Theorem~\ref{thm:normalizer2} is an immediate consequence of Theorem~\ref{thm:normalizer}; we need only explain the parts relating to disconnected union decompositions.

Let $\Xi$ denote the set of $G$-conjugacy classes of homomorphisms $N\to G$. We have a natural disjoint union $\Xi=\bigsqcup_{\lambda\in\Lambda^+}\Xi(\lambda)$, where $\Xi(\lambda)$ is the set of $G$-conjugacy classes containing some $\tau:N\to G$ whose restriction to $\Gg_m$ is $\lambda$.

For a given $\lambda\in\Lambda^+$, extending it to $\tau:N\to G$ amounts to choosing a (necessarily semisimple) element $\sigma=\tau([\begin{smallmatrix}0&1\\-1&0\end{smallmatrix}])\in G$ satisfying 
\begin{equation} \label{eqn:sigma-eqn}
\sigma^2=\lambda(-1)\;\text{ and }\;\sigma\lambda(z)\sigma^{-1}=\lambda(z)^{-1}\,\text{ for all }\,z\in\Gg_m.
\end{equation} 
Let $\Sigma(\lambda)$ denote the set of $\sigma\in G$ satisfying~\eqref{eqn:sigma-eqn}, which is a union of $Z_G(\lambda)$-conjugacy classes where $Z_G(\lambda)$ means the centralizer of the image of $\lambda$. Then $\Xi(\lambda)$ is in bijection with the set of $Z_G(\lambda)$-conjugacy classes in $\Sigma(\lambda)$. Note that $\Sigma(\lambda)$ and hence $\Xi(\lambda)$ are empty if $w_0\lambda\neq-\lambda$, i.e.\ $\lambda\notin\Lambda_1^+$.

Now suppose that $\lambda\in\Lambda_1^+$. Since the $Z_G(\lambda)$-conjugacy classes in $\Sigma(\lambda)$ are all closed, we have a disconnected union
\begin{equation} \label{eqn:disconnected}
\Bun_{G}^\lambda(\Aa^2/N)=\coprod_{\xi\in\Xi(\lambda)} \Bun_{G}^\xi(\Aa^2/N),
\end{equation} 
where $\Bun_{G}^\xi(\Aa^2/N)$ is the subvariety of $\Bun_G(\Aa^2/N)$ parametrizing $N$-equivariant pairs $(\cF,\Phi)$ for which the induced action of $N$ on the fibre of $\cF$ at $(0,0)\in\Aa^2$ gives rise to a homomorphism $N\to G$ in the $G$-conjugacy class $\xi$. By Theorem~\ref{thm:normalizer}, there is a corresponding disconnected union
\begin{equation} \label{eqn:disconnected2}
(\Gr_0^\lambda)^\iota=\coprod_{\xi\in\Xi(\lambda)} (\Gr_0)^{\iota,\xi},
\end{equation}
where $(\Gr_0)^{\iota,\xi}:=\Psi(\Bun_{G}^\xi(\Aa^2/N))$. We will give an alternative description of $(\Gr_0)^{\iota,\xi}$ in Proposition~\ref{prop:sigma} below. This completes the proof of Theorem~\ref{thm:normalizer2}.

If $\Xi(\lambda)$ is empty, then so are $\Bun_{G}^\lambda(\Aa^2/N)$ and $(\Gr_0^\lambda)^\iota$. (In the case $G=\SL(2)$, this provides another proof of Lemma~\ref{lem:even-empty}.) If $\Xi(\lambda)$ is nonempty, it is still possible that the varieties $\Bun_{G}^\xi(\Aa^2/N)$ and $(\Gr_0)^{\iota,\xi}$ are empty for some $\xi\in\Xi(\lambda)$. For example, this happens when $\lambda=0$ and $\xi$ corresponds to the $G$-conjugacy class of a nontrivial involution $\sigma\in G$. We do not know a general criterion for nonemptiness of $\Bun_{G}^\xi(\Aa^2/N)\cong(\Gr_0)^{\iota,\xi}$, but we will give one in the case $G=\SL(r)$ in Proposition~\ref{prop:glr}.

By analogy with the varieties $\Bun_G^\lambda(\Aa^2/\Gg_m)\cong\Gr_0^\lambda$, it is natural to suspect that the nonempty varieties $\Bun_{G}^\xi(\Aa^2/N)\cong(\Gr_0)^{\iota,\xi}$ are all connected. We will prove this in the case $G=\SL(r)$ in Proposition~\ref{prop:glr}. It does not seem to be easily deducible from the following alternative description of $(\Gr_0)^{\iota,\xi}$. 

\begin{prop} \label{prop:sigma}
Let $\lambda\in\Lambda_1^+$ and $\gamma\in(\Gr_0^\lambda)^\iota$.
\begin{enumerate} 
\item If we write $\gamma=q_1t^\lambda q_2$ for $q_1,q_2\in G[t]$, the element
\[
\sigma:=\left(t^\lambda(q_2|_{t\mapsto 0})(q_1|_{t\mapsto 0})t^\lambda\right)|_{t\mapsto 0}\in G
\]
is well defined and belongs to $\Sigma(\lambda)$. Moreover, its $Z_G(\lambda)$-conjugacy class depends only on $\gamma$, not on the choice of $q_1,q_2$. 
\item
Let $\xi\in\Xi(\lambda)$. Then $\gamma$ belongs to $(\Gr_0)^{\iota,\xi}$ if and only if $\sigma$ belongs to the $Z_G(\lambda)$-conjugacy class determined by $\xi$.
\end{enumerate}
\end{prop}

\begin{proof}
The assumption that $\iota(\gamma)=\gamma$ implies that
\begin{equation}
(q_1|_{t\mapsto -t})t^\lambda\lambda(-1)(q_2|_{t\mapsto -t})q_1t^\lambda q_2=1,
\end{equation}
and hence (after replacing $t$ by $tu$ and rearranging) 
\begin{equation} \label{eqn:tu}
t^\lambda (q_2|_{t\mapsto tu})(q_1|_{t\mapsto -tu}) t^\lambda=
u^{-\lambda}(q_1^{-1}|_{t\mapsto tu})(q_2^{-1}|_{t\mapsto -tu})u^{-\lambda}\lambda(-1).
\end{equation}
Now the left-hand side of~\eqref{eqn:tu} clearly belongs to $G[t,t^{-1},u]$ while the right-hand side clearly belongs to $G[t,u,u^{-1}]$. We conclude that their common value, $p$ say, in fact belongs to $G[t,u]$, so it make sense to define $\sigma=p|_{t,u\mapsto 0}\in G$. If we make the substitution $u\mapsto 0$ before the substitution $t\mapsto 0$, we obtain the expression for $\sigma$ given in the statement, showing that it is well defined.

We could prove the remaining assertions in (1) concerning $\sigma$ by simple calculations, but in order to prove (2) also it is more efficient to argue as follows. 

Recall from Proposition~\ref{prop:psi} the special quadruple $(\tilde{g}_{01}^{00},\tilde{g}_{10}^{00},\tilde{g}_{11}^{01},\tilde{g}_{11}^{10})\in Z$ in the $H$-orbit corresponding to $\gamma$ under $\Psi$. Since $\gamma$ is fixed by $\iota$, the $H$-orbit of $(\tilde{g}_{01}^{00},\tilde{g}_{10}^{00},\tilde{g}_{11}^{01},\tilde{g}_{11}^{10})$ is preserved by the action of $[\begin{smallmatrix}0&1\\-1&0\end{smallmatrix}]\in\SL(2)$. In fact, an easy calculation using~\eqref{eqn:diag-action} and~\eqref{eqn:switch-action} shows that
\begin{equation} \label{eqn:switch-fixed-special}
[\begin{smallmatrix}0&1\\-1&0\end{smallmatrix}]\cdot(\tilde{g}_{01}^{00},\tilde{g}_{10}^{00},\tilde{g}_{11}^{01},\tilde{g}_{11}^{10})=(p^{-1},1,1,\gamma_{t\mapsto -tu})\cdot(\tilde{g}_{01}^{00},\tilde{g}_{10}^{00},\tilde{g}_{11}^{01},\tilde{g}_{11}^{10}).
\end{equation}
The quadruple $(\tilde{g}_{01}^{00},\tilde{g}_{10}^{00},\tilde{g}_{11}^{01},\tilde{g}_{11}^{10})$ defines as usual a pair $(\cF,\Phi)\in\Bun_{G}^\lambda(\Aa^2/N)$. Let $\tau:N\to G$ be the homomorphism obtained from the action of $N$ on the fibre of $\cF$ at $(0,0)\in\Aa^2=U_{00}$, using the original identification of $\cF|_{U_{00}}$ with the trivial $G$-bundle on $U_{00}$. As already seen in the proof of Proposition~\ref{prop:identification}, the first component of~\eqref{eqn:z-fixed-special} says that the restriction of $\tau$ to $\Gg_m$ is $\lambda$. Likewise, the first component of~\eqref{eqn:switch-fixed-special} says that $\tau([\begin{smallmatrix}0&1\\-1&0\end{smallmatrix}])=p|_{t,u\mapsto 0}=\sigma$. From this it is automatic that $\sigma$ belongs to $\Sigma(\lambda)$, and that its $Z_G(\lambda)$-conjugacy class depends only on $\gamma$ (because the $G$-conjugacy class of $\tau$ depends only on the isomorphism class of $(\cF,\Phi)$). This completes part (1), and part (2) also follows by definition of $(\Gr_0)^{\iota,\xi}$.
\end{proof}

\begin{ex}
Continue Example~\ref{ex:sl2-special} with $G=\SL(2)$, $\lambda=\alpha$ (the positive coroot), and $\gamma=[\begin{smallmatrix}1&t^{-1}\\0&1\end{smallmatrix}]$. Using the same choice of $q_1$ and $q_2$ as in Example~\ref{ex:sl2-special}, the element $\sigma$ defined in Proposition~\ref{prop:sigma} is $[\begin{smallmatrix}0&1\\-1&0\end{smallmatrix}]$. In this case, $\Sigma(\alpha)$ consists of a single $Z_G(\alpha)$-conjugacy class, so $\Xi(\alpha)$ has only one element.
\end{ex}

\begin{ex}
More generally, for any $G$ suppose that $\gamma=e(x)$ for some $x\in\cN$. As seen in~\eqref{eqn:exp-inclusion}, $\gamma\in(\Gr_0^{\lambda_{\cO}})^\iota$ where $\cO$ is the $G$-orbit of $x$. If we let $\varphi:\SL(2)\to G$ be a homomorphism such that $(d\varphi)([\begin{smallmatrix}0&1\\0&0\end{smallmatrix}])=x$, then applying $\varphi$ to the previous example shows that the element $\sigma$ defined in Proposition~\ref{prop:sigma} is $\varphi([\begin{smallmatrix}0&1\\-1&0\end{smallmatrix}])$. So Proposition~\ref{prop:sigma} says that $e(x)\in(\Gr_0)^{\iota,\xi}$ where $\xi$ is the $G$-conjugacy class of the restriction of $\varphi$ to $N$. This was already clear from the definition $(\Gr_0)^{\iota,\xi}=\Psi(\Bun_{G}^\xi(\Aa^2/N))$: it follows from the commutativity of the diagram~\eqref{eqn:intro-diag}.
\end{ex}

\begin{rmk}
As explained in~\cite[Section 5.1]{bf}, the varieties $\Bun_{G}^\lambda(\Aa^2/\Gg_m)$ are special cases of certain moduli spaces of $\Gamma_{\A_k}$-equivariant principal $G$-bundles on $\Pp^2$, where $\Gamma_{\A_k}\cong\Z/(k+1)\Z$ is a finite cyclic subgroup of $\SL(2)$ (alternatively, moduli spaces of `instantons on a type-$\A_k$ singularity'). This is because $\Gg_m$ contains a copy of $\Gamma_{\A_k}$ for all $k$, and is Zariski-generated by the union of these finite subgroups; one should think of $\Gg_m$ as the type-$\A_\infty$ subgroup of $\SL(2)$. Analogously, $N$ is the type-$\D_\infty$ subgroup of $\SL(2)$; see \S\ref{ss:subgroups}. Hence the varieties $\Bun_G^\xi(\Aa^2/N)$ are special cases of certain moduli spaces of $\Gamma_{\D_k}$-equivariant principal $G$-bundles on $\Pp^2$, where $\Gamma_{\D_k}$ is a binary dihedral subgroup of $\SL(2)$ (`instantons on a type-$\D_k$ singularity'). It would be interesting to extend the results of~\cite{bf} to the latter moduli spaces.
\end{rmk}

\subsection{Principal bundles on $\Pp^2$}
\label{ss:p2}

For the remainder of the paper it will be more convenient to use the definition~\eqref{eqn:old-def} of $\Bun_G(\Aa^2)$ rather than the definition~\eqref{eqn:new-def}, so we want to translate our explicit formula for the bijection $\Psi$ (given in Proposition~\ref{prop:psi}) to the setting of~\eqref{eqn:old-def}. 

As explained in~\cite[Section 3]{atiyah} and~\cite[Section 4.1]{bfg}, the two definitions are linked using the diagram
\begin{equation}
\vcenter{
\xymatrix@R=10pt@C=10pt{
&X\ar[dl]_-{\pi_1}\ar[dr]^-{\pi_2}&\\
\Pp^2&&\Pp^1\times\Pp^1
}}
\end{equation}
where 
\[
X:=\{([z_0:z_1:z_2],\,t,\,u)\in\Pp^2\times\Pp^1\times\Pp^1\,|\,z_1=z_0t,\,z_2=z_0u\},
\]
and $\pi_1$, $\pi_2$ are the obvious projections. In the definition of $X$, we have used homogeneous coordinates $[z_0:z_1:z_2]$ on the $\Pp^2$ factor and inhomogeneous coordinates $t$ and $u$ (as before) on the $\Pp^1$ factors. One must interpret the equations accordingly: for example, if $t=\infty$ then the equation $z_1=z_0t$ should be read as $z_0=0$. Thus, the map $\pi_1$ is a blow-up of $\Pp^2$ at the two points $[0:0:1]$ and $[0:1:0]$ on the line at infinity, and the map $\pi_2$ is a blow-up of $\Pp^1\times\Pp^1$ at the point $(\infty,\infty)$. Neither blow-up affects the open subset $\Aa^2$, so the pull-backs give isomorphisms
\begin{equation}
\Bun_G(\Pp^2;\ell_\infty)\simto\Bun_G(X;X\smallsetminus\Aa^2)\overset{\sim}{\leftarrow}
\Bun_G(\Pp^1\times\Pp^1;D_\infty),
\end{equation} 
where $\ell_\infty=\Pp^2\smallsetminus\Aa^2$ is the line at infinity.

We can present $\Bun_G(\Pp^2;\ell_\infty)$ as a quotient in much the same way as~\eqref{eqn:p1p1presentation}. We use the usual open covering of $\Pp^2$:
\[
\begin{split}
U_0&:=\{[z_0:z_1:z_2]\,|\,z_0\neq 0\}=\Aa^2,\\
U_1&:=\{[z_0:z_1:z_2]\,|\,z_1\neq 0\},\\
U_2&:=\{[z_0:z_1:z_2]\,|\,z_2\neq 0\}.
\end{split}
\] 
A pair $(\cF,\Phi)\in\Bun_G(\Pp^2;\ell_\infty)$ can be constructed by gluing together trivial bundles on these open sets using transition functions $(g_1^0,g_2^0,g_2^1)$, where
\begin{equation} \label{eqn:p2conditions}
\begin{split}
g_1^0&\in G[z_1/z_0,z_0/z_1,z_2/z_0],\\
g_2^0&\in G[z_1/z_0,z_2/z_0,z_0/z_2],\\
g_2^1&\in \ker(G[z_0/z_1,z_2/z_1,z_1/z_2]\to G[z_2/z_1,z_1/z_2]),\\
g_2^0&=g_2^1g_1^0\quad\text{ in }G(z_0,z_1,z_2).
\end{split}
\end{equation}
Let $Z'$ denote the space of triples $(g_1^0,g_2^0,g_2^1)$ satisfying~\eqref{eqn:p2conditions}. On this space we have an action of the group
\begin{equation}
\begin{split}
H':=&G[z_1/z_0,z_2/z_0]\\
&\times \ker(G[z_0/z_1,z_2/z_1]\to G[z_2/z_1]) \\
&\times \ker(G[z_0/z_2,z_1/z_2]\to G[z_1/z_2]),
\end{split}
\end{equation}
defined by
\begin{equation}\label{eqn:p2action}
(h_0,h_1,h_2)\cdot (g_1^0,g_2^0,g_2^1) = (h_1g_1^0h_0^{-1},h_2g_2^0h_0^{-1},h_2g_2^1h_1^{-1}),
\end{equation}
for $(h_0,h_1,h_2)\in H'$ and $(g_1^0,g_2^0,g_2^1)\in Z'$. Then
\begin{equation} \label{eqn:p2presentation}
\Bun_G(\Aa^2)=\Bun_G(\Pp^2;\ell_\infty)\cong Z'/H',
\end{equation} 
and Proposition~\ref{prop:free} implies that the $H'$-action on $Z'$ is free.

In this description, the action of $G$ on $\Bun_G(\Aa^2)$ is induced by the following action on $Z'$:
\begin{equation} \label{eqn:p2-G-action}
g\cdot(g_1^0,g_2^0,g_2^1)=(gg_1^0,gg_2^0,gg_2^1g^{-1}).
\end{equation}
The action of the diagonal subgroup of $\GL(2)$ on $\Bun_G(\Aa^2)$ is induced by the following action on $Z'$:
\begin{equation} \label{eqn:p2-diag-action}
[\begin{smallmatrix}\alpha&0\\0&\beta\end{smallmatrix}]\cdot(g_1^0,g_2^0,g_2^1)
=(g_1^0,g_2^0,g_2^1)\left|_{\substack{z_1/z_0\mapsto \alpha^{-1}z_1/z_0\\z_2/z_0\mapsto \beta^{-1}z_2/z_0}}\right.\, .
\end{equation}
As before, we identify $\Bun_G(\Aa^2/\Gg_m)$ with the fixed-point set of $\Gg_m$ on $\Bun_G(\Aa^2)$, i.e.\ the set of $H'$-orbits in $Z'$ that are fixed by the action of $\Gg_m$.

The bijection $\Psi:\Bun_G(\Aa^2/\Gg_m)\to\Grom$, or more conveniently its inverse, is described as follows.
\begin{prop} \label{prop:p2-psi}
If $\gamma\in\Gr_0^\lambda$ for $\lambda\in\Lambda^+$, i.e.\ 
\[
\gamma=q_1 t^\lambda q_2\;\text{ for some }\,q_1,q_2\in G[t],
\]
then $\Psi^{-1}(\gamma)$ is the $H'$-orbit of the triple $(\tilde{g}_{1}^{0},\tilde{g}_{2}^{0},\tilde{g}_{2}^{1})\in Z'$ defined by
\[
\begin{split}
\tilde{g}_{1}^{0}&=(q_2^{-1}|_{t\mapsto z_1z_2/z_0^2})\,(z_1/z_0)^{-\lambda},\\
\tilde{g}_{2}^{0}&=(q_1|_{t\mapsto z_1z_2/z_0^2})\,(z_2/z_0)^\lambda,\\
\tilde{g}_{2}^{1}&=\gamma|_{t\mapsto z_1z_2/z_0^2}.
\end{split}
\]
\end{prop}

\begin{ex} \label{ex:sl2-special-p2}
Continue Example~\ref{ex:sl2-special} with $G=\SL(2)$ and $\gamma=[\begin{smallmatrix}1&t^{-1}\\0&1\end{smallmatrix}]\in\Gr_0^\alpha$. Choosing $q_1$ and $q_2$ as in Example~\ref{ex:sl2-special}, the special representative of $\Psi^{-1}(\gamma)$ specified in Proposition~\ref{prop:p2-psi} is $([\begin{smallmatrix}z_0/z_1&0\\-z_2/z_0&z_1/z_0\end{smallmatrix}],[\begin{smallmatrix}0&z_0/z_2\\-z_2/z_0&z_1/z_0\end{smallmatrix}],[\begin{smallmatrix}1&z_0^2/z_1z_2\\0&1\end{smallmatrix}])\in Z'$.
\end{ex}

\begin{proof}
Let $(\cF_1,\Phi_1)\in\Bun_G(\Pp^2;\ell_\infty)$ be determined by the triple $(\tilde{g}_{1}^{0},\tilde{g}_{2}^{0},\tilde{g}_{2}^{1})\in Z'$ defined as in Proposition~\ref{prop:p2-psi}, and let $(\cF_2,\Phi_2)\in\Bun_G(\Pp^1\times\Pp^1;D_\infty)$ be determined by the quadruple $(\tilde{g}_{01}^{00},\tilde{g}_{10}^{00},\tilde{g}_{11}^{01},\tilde{g}_{11}^{10})\in Z$ defined as in Proposition~\ref{prop:psi} (using the same expression $\gamma=q_1 t^\lambda q_2$). It suffices to show that the pull-backs $\pi_1^*(\cF_1,\Phi_1)$ and $\pi_2^*(\cF_2,\Phi_2)$ give the same point of $\Bun_G(X;X\smallsetminus\Aa^2)$. We will only consider the bundles $\pi_1^*\cF_1$ and $\pi_2^*\cF_2$; it is easy to identify the two trivializations on $X\smallsetminus\Aa^2$.

Now $\pi_1^*\cF_1$ can be obtained by gluing trivial $G$-bundles on $\pi_1^{-1}(U_0)$, $\pi_1^{-1}(U_1)$, and $\pi_1^{-1}(U_2)$ using the transition functions $\pi_1^*\tilde{g}_{1}^{0},\pi_1^*\tilde{g}_{2}^{0},\pi_1^*\tilde{g}_{2}^{1}$ on the respective overlaps $\pi_1^{-1}(U_0\cap U_1)$, $\pi_1^{-1}(U_0\cap U_2)$, $\pi_1^{-1}(U_1\cap U_2)$. Note that $\pi_1$ induces an isomorphism $\pi_1^{-1}(U_0\cap U_1)\simto U_0\cap U_1$, and $\pi_1^*\tilde{g}_{1}^{0}$ is the morphism $\pi_1^{-1}(U_0\cap U_1)\to G$ obtained by composing this isomorphism with $\tilde{g}_{1}^{0}$; and similarly for the other overlaps.

We have an analogous description of $\pi_2^*\cF_2$. Since $\tilde{g}_{11}^{01}=1$, $\pi_2^*\cF_2$ is obtained by gluing trivial $G$-bundles on $\pi_2^{-1}(U_{00})$, $\pi_2^{-1}(U_{10})$, and $\pi_2^{-1}(U_{01}\cup U_{11})$ using the transition functions $\pi_2^*\tilde{g}_{10}^{00},\pi_2^*\tilde{g}_{01}^{00},\pi_2^*\tilde{g}_{11}^{10}$ on the respective overlaps 
\[
\begin{split}
&\pi_2^{-1}(U_{00}\cap U_{10}),\\ 
&\pi_2^{-1}(U_{00}\cap(U_{01}\cup U_{11}))=\pi_2^{-1}(U_{00}\cap U_{01}),\\
&\pi_2^{-1}(U_{10}\cap (U_{01}\cup U_{11}))=\pi_2^{-1}(U_{10}\cap U_{11}).
\end{split}
\] 

It now suffices to check that the two collections of transition functions agree on dense subsets of their domains; and since $X$ is irreducible, it suffices to check the agreement on generic points. For a generic point $([z_0:z_1:z_2],t,u)$ of $X$, we have
\[ 
\begin{split}
\pi_1^*\tilde{g}_{1}^{0}([z_0:z_1:z_2],t,u)&=(q_2^{-1}|_{t\mapsto z_1z_2/z_0^2})\,(z_1/z_0)^{-\lambda}\\
&=(q_2^{-1}|_{t\mapsto tu})\,t^{-\lambda}=\pi_2^*\tilde{g}_{10}^{00}([z_0:z_1:z_2],t,u),
\end{split}
\]
and the arguments for the other transition functions are analogous.
\end{proof}


\section{The case $G=\SL(r)$}
\label{sec:glr}


In this section we assume that $G=\SL(r)$ for some integer $r\geq 2$. In the definition~\eqref{eqn:old-def} of $\Bun_G(\Aa^2)$, we can replace the principal $G$-bundle $\cF$ on $\Pp^2$ by a rank-$r$ vector bundle $\cE$ on $\Pp^2$. (Note that we do not need to impose the condition that $\cE$ has trivial determinant bundle, since the existence of the trivialization $\Phi$ on $\ell_\infty$ implies this automatically.) We then have ADHM-style or `monad' descriptions of all our moduli spaces as suitable kinds of Nakajima quiver varieties. We will recall these descriptions and see what they tell us about the varieties $\Gr_0^\lambda$ and $(\Gr_0)^{\iota,\xi}$ in this case. For $\Gr_0^\lambda$, this recovers a known result of Mirkovi\'c and Vybornov~\cite{mvy-cr}, as already observed in~\cite{bf}. The application to the varieties $(\Gr_0)^{\iota,\xi}$ is all that is actually new here, but it is easiest to explain in a more general context.

\subsection{Vector bundles on $\Pp^2$}
\label{ss:1-case}

When $G=\SL(r)$ for $r\geq 2$, the moduli space $\Bun_{G}(\Aa^2)=\Bun_{G}(\Pp^2;\ell_\infty)$ has connected components $\Bun_{G}^n(\Aa^2)$ where $n$ is a nonnegative integer representing the second Chern class of the bundle. As explained in~\cite[Theorem 2.1]{nak3},~\cite[Section 2]{vv} and~\cite[Section 5.1]{bfg}, $\Bun_{G}^n(\Aa^2)$ is isomorphic to the quiver variety $\fM_0^{\mathrm{reg}}(n,r)$ of type $\widetilde{\A}_0$ (that is, corresponding to the quiver with one vertex and one loop). We now recall the definitions of the latter variety and of the isomorphism, mainly following the conventions of~\cite{nak1,nak2,nak3}. 

Set $V=\C^n$ and $W=\C^r$, so that $G=\SL(W)$. Let $\Lambda(V,W)$ be the variety of quadruples $(B_1,B_2,\bi,\bj)$, where
\[
B_1,B_2:V\to V,\quad \bi:W\to V,\quad \bj:V\to W
\]    
are linear maps satisfying the ADHM equation
\begin{equation} \label{eqn:adhm}
[B_1,B_2]+\bi\bj=0.
\end{equation}
We have a natural action of $\GL(V)\times G$ on $\Lambda(V,W)$. 

A quadruple $(B_1,B_2,\bi,\bj)\in\Lambda(V,W)$ is said to be \emph{stable} if $V$ has no nonzero subspace which is preserved by $B_1,B_2$ and is contained in the kernel of $\bj$, and \emph{costable} if $V$ has no proper subspace which is preserved by $B_1,B_2$ and contains the image of $\bi$. Let $\Lambda(V,W)^{s}$ (respectively, $\Lambda(V,W)^{sc}$) be the open subset of $\Lambda(V,W)$ consisting of quadruples that are stable (respectively, stable and costable); these subsets are clearly preserved by the action of $\GL(V)\times G$. 

It follows by the same arguments as in~\cite[Section 3]{nak2} that:
\begin{equation} \label{eqn:stability-freeness}
\text{The action of $\GL(V)$ on $\Lambda(V,W)^{s}$ is free.}
\end{equation}
So there is a geometric quotient variety $\fM(V,W)=\Lambda(V,W)^s/\GL(V)$, whose points are identified with the $\GL(V)$-orbits in $\Lambda(V,W)^s$. We let $\fM_0(V,W)$ denote the affine variety $\Lambda(V,W)/\!/\GL(V)$, whose points are ientified with the closed $\GL(V)$-orbits in $\Lambda(V,W)$. We have an action of $G$ on both $\fM(V,W)$ and $\fM_0(V,W)$, and there is a natural $G$-equivariant projective morphism 
\[ \pi:\fM(V,W)\to\fM_0(V,W), \]
mapping a $\GL(V)$-orbit $\cO$ in $\Lambda(V,W)^s$ to the unique closed $\GL(V)$-orbit in $\overline{\cO}$, where $\overline{\cO}$ is the closure of $\cO$ in $\Lambda(V,W)$.

We define $\fM^{\mathrm{reg}}(V,W)$ to be the $G$-stable open subvariety $\Lambda(V,W)^{sc}/\GL(V)$ of $\fM(V,W)$. The morphism $\pi$ maps $\fM^{\mathrm{reg}}(V,W)$ isomorphically onto an open subvariety $\fM_0^{\mathrm{reg}}(V,W)$ of $\fM_0(V,W)$. In fact, the $\Gamma=\{1\}$ case of~\cite[Lemma 1(ii)]{vv} says that $\Lambda(V,W)^{sc}$ consists exactly of the points in $\Lambda(V,W)^s$ whose $\GL(V)$-orbit is closed in $\Lambda(V,W)$, so the points of either $\fM^{\mathrm{reg}}(V,W)$ or $\fM_0^{\mathrm{reg}}(V,W)$ are identified with the closed $\GL(V)$-orbits in $\Lambda(V,W)$ that lie in $\Lambda(V,W)^s$. The notation $\fM_0^{\mathrm{reg}}(V,W)$ is the one used by Nakajima, but we will mostly refer to $\fM^{\mathrm{reg}}(V,W)$.

The following result, essentially due to Barth~\cite{barth,donaldson}, is part of a similar moduli-space interpretation of the whole of $\fM(V,W)$; see~\cite[Theorem 2.1]{nak3}.
\begin{prop} \label{prop:barth}
We have a $G$-equivariant isomorphism 
\[ 
\Theta:\fM^{\mathrm{reg}}(V,W)\simto\Bun_G^n(\Aa^2)
\] 
which sends the $\GL(V)$-orbit of a quadruple $(B_1,B_2,\bi,\bj)\in\Lambda(V,W)^{sc}$ to the point parametrizing the rank-$r$ vector bundle $\ker(b)/\mathrm{im}(a)$ on $\Pp^2$ where the maps
\[
V\otimes \cO_{\Pp^2}(-1) \overset{a}{\hookrightarrow}
(V\oplus V\oplus W)\otimes \cO_{\Pp^2} \overset{b}{\twoheadrightarrow}
V\otimes \cO_{\Pp^2}(1)
\]
are defined by
\[
a=\begin{bmatrix}z_0B_1-z_1\mathrm{id}_V\\z_0B_2-z_2\mathrm{id}_V\\z_0 \bj \end{bmatrix},\qquad
b=\begin{bmatrix}-(z_0B_2-z_2\mathrm{id}_V) & z_0B_1-z_1\mathrm{id}_V & z_0 \bi \end{bmatrix}.
\]
Here $z_0,z_1,z_2\in\Gamma(\Pp^2,\cO_{\Pp^2}(1))$ are homogeneous coordinates as before.
\end{prop}
\noindent
Some explanation is required. The equation~\eqref{eqn:adhm} ensures that $b\circ a=0$, so $\ker(b)/\mathrm{im}(a)$ makes sense. The condition that $(B_1,B_2,\bi,\bj)$ is stable and costable is equivalent to saying that $a$ is injective and $b$ is surjective, not just in the category of vector bundles on $\Pp^2$ but on every fibre; for this, combine~\cite[Lemma 2.7]{nak3} with its dual. When we restrict to the line at infinity $\ell_\infty=\Pp^2\smallsetminus\Aa^2$ where $z_0=0$, the maps $a$ and $b$ become
\[
\begin{bmatrix}-z_1\mathrm{id}_V\\-z_2\mathrm{id}_V\\0 \end{bmatrix}\ \text{ and }\ 
\begin{bmatrix}z_2\mathrm{id}_V & -z_1\mathrm{id}_V & 0\end{bmatrix},
\] 
so the restriction of $\ker(b)/\mathrm{im}(a)$ to $\ell_\infty$ is isomorphic in an obvious way to the trivial rank-$r$ vector bundle $W\otimes\cO_{\ell_\infty}$; this is the trivialization used to regard $\ker(b)/\mathrm{im}(a)$ as a point of $\Bun_G(\Aa^2)$. From this the $G$-equivariance of $\Theta$ is obvious. 

On $\Bun_G^n(\Aa^2)$, as we saw in the previous section, we have an action not just of $G$ but of $G\times\GL(2)$, where $[\begin{smallmatrix}\alpha&\beta\\\gamma&\delta\end{smallmatrix}]\in\GL(2)$ acts as the pull-back under the action of $[\begin{smallmatrix}\alpha&\beta\\\gamma&\delta\end{smallmatrix}]^{-1}$ on $\Pp^2$. The corresponding $\GL(2)$-action on $\fM^{\mathrm{reg}}(V,W)$ is described as follows. We have a $\GL(2)$-action on $\Lambda(V,W)$ by the rule
\begin{equation} \label{eqn:gl2-action}
[\begin{smallmatrix}\alpha&\beta\\\gamma&\delta\end{smallmatrix}]\cdot(B_1,B_2,\bi,\bj)
=(\alpha B_1+\gamma B_2,\beta B_1+\delta B_2,(\alpha\delta-\beta\gamma)\bi,\bj).
\end{equation}
More intrinsically, the $\GL(2)$-action is defined by regarding $B_1,B_2$ as the components of a single linear map $B:V\to V\otimes\C^2$, and $\bi$ as a map $W\to V\otimes\Lambda^2(\C^2)$, and then letting $\GL(2)$ act via its action on $\C^2$; this formulation is from~\cite[Section 2]{vv}. 
Since this $\GL(2)$-action on $\Lambda(V,W)$ commutes with the action of $\GL(V)\times G$ and preserves the subset $\Lambda(V,W)^{s}$ (respectively, $\Lambda(V,W)^{sc}$), it induces a $\GL(2)$-action on $\fM(V,W)$ (respectively, $\fM^{\mathrm{reg}}(V,W)$) which commutes with the action of $G$.

\begin{prop} \label{prop:barth-equiv}
The isomorphism $\Theta:\fM^{\mathrm{reg}}(V,W)\simto\Bun_G^n(\Aa^2)$ is $\GL(2)$-equivariant, for the $\GL(2)$-actions described above.
\end{prop}

\begin{proof}
This is implicit in~\cite{nak3}, and various special cases which are no easier than the full claim were explicitly proved in~\cite[Section 3]{furutahashimoto},~\cite[Theorem 1]{vv} and~\cite[Lemma 4.3]{kumar}. For reassurance we spell out the proof. Suppose that $(B_1,B_2,\bi,\bj)\in\Lambda(V,W)^{sc}$ gives rise to the vector bundle $\ker(b)/\mathrm{im}(a)$ as in Proposition~\ref{prop:barth}. Then by definition, the pull-back under the action of $[\begin{smallmatrix}\alpha&\beta\\\gamma&\delta\end{smallmatrix}]$ on $\Pp^2$ of the vector bundle corresponding to $[\begin{smallmatrix}\alpha&\beta\\\gamma&\delta\end{smallmatrix}]\cdot(B_1,B_2,\bi,\bj)$ is $\ker(b')/\mathrm{im}(a')$, where
\[
a'=\begin{bmatrix}\alpha&\beta&0\\\gamma&\delta&0\\0&0&1\end{bmatrix}a\qquad \text{ and }\qquad
b'=(\alpha\delta-\beta\gamma)\,b\begin{bmatrix}\alpha&\beta&0\\\gamma&\delta&0\\0&0&1\end{bmatrix}^{-1}.
\]
Clearly we have an isomorphism between $\ker(b)/\mathrm{im}(a)$ and $\ker(b')/\mathrm{im}(a')$ respecting the trivializations on $\ell_\infty$.
\end{proof}

As an immediate consequence of Proposition~\ref{prop:barth-equiv}, the isomorphism $\Theta$ of Proposition~\ref{prop:barth} restricts to an isomorphism between fixed-point subvarieties
\begin{equation} \label{eqn:barth-gamma}
\fM^{\mathrm{reg}}(V,W)^\Gamma\simto\Bun_G^n(\Aa^2/\Gamma),
\end{equation}
for any subgroup $\Gamma\subseteq\GL(2)$, where $\Bun_G^n(\Aa^2/\Gamma):=\Bun_G(\Aa^2/\Gamma)\cap\Bun_G^n(\Aa^2)$. (We continue to rely on Proposition~\ref{prop:free} to identify $\Bun_G(\Aa^2/\Gamma)$ with $\Bun_G(\Aa^2)^\Gamma$.) Our aim now is to analyse the left-hand side of~\eqref{eqn:barth-gamma}.

\subsection{Reductive subgroups of $\SL(2)$}
\label{ss:subgroups}
  
We need to digress briefly to recall some aspects of the McKay correspondence, and how they extend from the well-known case of finite subgroups of $\SL(2)$ to the (slightly) more general case of reductive subgroups of $\SL(2)$.

Let $\Gamma$ be a reductive subgroup of $\SL(2)$. Then $\Gamma$ is either finite (hence cyclic or binary di-, tetra-, octa- or icosahedral), conjugate to the diagonal $\Gg_m$ or to $N=N_{\SL(2)}(\Gg_m)$, or equal to $\SL(2)$ itself. 

We define the \emph{doubled McKay quiver} of $\Gamma$ as follows: its vertex set $I_\Gamma$ is in bijection with the isomorphism classes of irreducible representations $\{S_i\,|\,i\in I_\Gamma\}$ of $\Gamma$, and the number of directed edges from vertex $i$ to vertex $j$ is the multiplicity $\dim \Hom_\Gamma(S_i\otimes\C^2,S_j)$, where $\Gamma$ acts on $\C^2$ via the embedding $\Gamma\hookrightarrow\SL(2)$. Since $\C^2$ is a self-dual representation of $\Gamma$, the number of edges from $j$ to $i$ equals that from $i$ to $j$. Moreover, when $\Gamma\neq\{1\}$ we have no edges from a vertex to itself. Hence the doubled McKay quiver is always obtained from an undirected graph, the \emph{McKay graph}, by replacing each undirected edge with two directed eges in opposite directions. The McKay graph is a simple graph unless $|\Gamma|\in\{1,2\}$. In fact:
\begin{itemize}
\item When $\Gamma=\{1\}$, the McKay graph is the `affine Dynkin diagram of type $\widetilde{\A}_0$', i.e.\ it has a single vertex and a single edge which is a loop joining that vertex to itself.
\item (The \emph{McKay correspondence} of~\cite{mckay}.) When $\Gamma$ is cyclic of order $k+1$ for some $k\geq 1$, respectively binary dihedral of order $4(k-2)$ for some $k\geq 4$, respectively binary tetrahedral, binary octahedral or binary icosahedral, the McKay graph is the affine Dynkin diagram of type $\widetilde{\A}_k$, respectively $\widetilde{\D}_k$, respectively $\widetilde{\E}_6$, $\widetilde{\E}_7$ or $\widetilde{\E}_8$. (The affine Dynkin diagram of type $\widetilde{\A}_1$ has two vertices and two edges joining them.) One says that $\Gamma$ itself is of type $\A_k$, respectively $\D_k$, respectively $\E_6$, $\E_7$ or $\E_8$.
\item When $\Gamma=\Gg_m$, the McKay graph is the Dynkin diagram of type $\A_\infty$ (infinite in both directions), with vertex set $\Z$. Here $S_i$, $i\in\Z$, is the one-dimensional representation $z\mapsto z^i$.
\item When $\Gamma=N$, the McKay graph is the Dynkin diagram of type $\D_\infty$, with vertex set $\{(0,+),(0,-)\}\cup\Z^+$ and edges as follows:
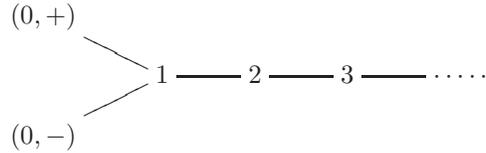
\[
\vcenter{
\xymatrix@R=2ex{
(0,+)\ar@{-}[dr]\\
&1\ar@{-}[r]&2\ar@{-}[r]&3\ar@{-}[r]&\cdot\cdots\cdot\\
(0,-)\ar@{-}[ur]
}
}
\]
Here $S_{0,+}$ is the trivial representation, $S_{0,-}$ is the nontrivial one-dimensional representation (in which $\Gg_m$ acts trivially and $[\begin{smallmatrix}0&1\\-1&0\end{smallmatrix}]$ acts by $-1$), and $S_i$ for $i>0$ is the two-dimensional irreducible representation whose restriction to $\Gg_m$ is the direct sum of the representations labelled by $i$ and $-i$.
\item When $\Gamma=\SL(2)$, the McKay graph is the Dynkin diagram of type $\A_{+\infty}$ (infinite in one direction), with vertex set $\N$. Here $S_i$, $i\in\N$, is the $(i+1)$-dimensional irreducible representation.   
\end{itemize} 

\begin{rmk}
It is an observation going back to Slodowy~\cite[Section 6.2, Appendix III]{slod} that certain inclusions of finite subgroups of $\SL(2)$ induce particularly nice relationships between their McKay graphs. For example, in the terminology introduced by the author and A.~Licata in~\cite{hl}, the fact that a binary dihedral group of order $4(k-2)$ contains an index-$2$ cyclic subgroup of order $2(k-2)$ is reflected in the fact that the affine Dynkin diagram of type $\widetilde{\D}_k$ can be obtained by applying the \emph{split-quotient} construction to a suitable graph involution of the affine Dynkin diagram of type $\widetilde{\A}_{2k-5}$; see~\cite[Example 2.7]{hl}. The same principles also apply to the infinite reductive subgroups of $\SL(2)$: the fact that $N$ contains $\Gg_m$ as an index-$2$ subgroup is reflected in the fact that the diagram of type $\D_\infty$ can be obtained by applying the split-quotient construction to the involution $i\mapsto -i$ of the diagram of type $\A_\infty$ (infinite in both directions).
\end{rmk} 

In the discussion of quiver varieties that follows, we will need to use certain results that were originally stated only for finite subgroups $\Gamma$ of $\SL(2)$ (for example in~\cite{hl,nak1,nak2,vv}). The relevant results all extend to the case of infinite reductive subgroups $\Gamma$ of $\SL(2)$ with the same proof. 
  
\subsection{Quiver varieties}
\label{ss:quiver-varieties}

Continue to let $\Gamma$ be a reductive subgroup of $\SL(2)$. Suppose that we are given a representation $\rho:\Gamma\to\GL(V)$ of $\Gamma$ on $V=\C^n$. We could also consider a representation of $\Gamma$ on $W=\C^r$, but for simplicity we assume that this representation is trivial, which agrees with the $\mu=0$ case we have considered elsewhere in the paper (see Remark~\ref{rmk:slice}).

Let $i:\Gamma\hookrightarrow\GL(2)$ be the inclusion. Then $\Gamma$ acts on $\Lambda(V,W)$ via the homomorphism $(\rho,i):\Gamma\to\GL(V)\times\GL(2)$, and we can consider the fixed-point subvariety $\Lambda(V,W)^{\Gamma,\rho}$. This carries an action of $\GL_{\Gamma,\rho}(V)\times G$, where $\GL_{\Gamma,\rho}(V)\subseteq\GL(V)$ is the automorphism group of the representation $\rho$. Let $\Lambda(V,W)^{s,\Gamma,\rho}$, $\Lambda(V,W)^{sc,\Gamma,\rho}$ denote the intersections of $\Lambda(V,W)^{\Gamma,\rho}$ with $\Lambda(V,W)^{s}$, $\Lambda(V,W)^{sc}$. We define
\[
\begin{split}
\fM^{\Gamma,\rho}(V,W)&:=\Lambda(V,W)^{s,\Gamma,\rho}/\GL_{\Gamma,\rho}(V),\\  
\fM^{\Gamma,\rho,\mathrm{reg}}(V,W)&:=\Lambda(V,W)^{sc,\Gamma,\rho}/\GL_{\Gamma,\rho}(V).
\end{split}
\]
Then $\fM^{\Gamma,\rho,\mathrm{reg}}(V,W)$ is an open subvariety of $\fM^{\Gamma,\rho}(V,W)$; for general $\rho$, either the former variety or both varieties could be empty (that is, $\Lambda(V,W)^{\Gamma,\rho}$ could have no intersection with $\Lambda(V,W)^{sc}$ or even with $\Lambda(V,W)^{s}$). 

One can also define $\fM_0^{\Gamma,\rho}(V,W):=\Lambda(V,W)^{\Gamma,\rho}/\GL_{\Gamma,\rho}(V)$ and a projective morphism $\pi:\fM^{\Gamma,\rho}(V,W)\to \fM_0^{\Gamma,\rho}(V,W)$, in the same way as in the $\Gamma=\{1\}$ case which was considered in \S\ref{ss:1-case}. The comments in \S\ref{ss:1-case} about $\fM^{\mathrm{reg}}(V,W)$ versus $\fM_0^{\mathrm{reg}}(V,W)$ apply equally well to $\fM^{\Gamma,\rho,\mathrm{reg}}(V,W)$.

The varieties $\fM^{\Gamma,\rho}(V,W)$ are examples of \emph{Nakajima quiver varieties}; for short, we will refer to the varieties $\fM^{\Gamma,\rho,\mathrm{reg}}(V,W)$ as quiver varieties also.
As Nakajima observed in~\cite{nak1}, by splitting up the representation $(V,\rho)$ into its isotypic components, one can regard the $B$ components of a point $(B_1,B_2,\bi,\bj)\in\Lambda(V,W)^{\Gamma,\rho}$ as a configuration of linear maps $(B_h)$ assigned to the edges $h$ of the doubled McKay quiver of $\Gamma$, and an element of $\GL_{\Gamma,\rho}(V)$ as a tuple of invertible linear transformations assigned to the vertices of the McKay graph of $\Gamma$. Thus the above definition of $\fM^{\Gamma,\rho}(V,W)$, which comes from~\cite[Section 2]{vv}, becomes a special case of the usual definition of quiver variety from~\cite{nak1,nak2}. We will make the latter definition explicit in the $\Gamma=\Gg_m$ and $\Gamma=N$ cases in \S\ref{ss:Gm-case} and \S\ref{ss:N-case} respectively.     

The following is a special case of a general result about quiver varieties, due to Nakajima and Crawley-Boevey:

\begin{prop} \label{prop:connected}
Let $\rho:\Gamma\to\GL(V)$ be a representation. If $\fM^{\Gamma,\rho}(V,W)$ is nonempty, it is nonsingular and connected. The same holds for $\fM^{\Gamma,\rho,\mathrm{reg}}(V,W)$.
\end{prop} 

\begin{proof}
Since $\fM^{\Gamma,\rho,\mathrm{reg}}(V,W)$ is an open subvariety of $\fM^{\Gamma,\rho}(V,W)$, it suffices to prove the statements about $\fM^{\Gamma,\rho}(V,W)$. The nonsingularity statement is~\cite[Corollary 3.12]{nak2}, and the connectedness statement is proved in~\cite[Section 1]{c-b}. 
\end{proof}

Because of the isomorphism~\eqref{eqn:barth-gamma}, we are interested in the fixed-point subvariety $\fM(V,W)^\Gamma$, whose definition does not involve $\rho$ (only the inclusion $i:\Gamma\hookrightarrow\GL(2)$). The following result was implicit in~\cite{nak1,vv} and appeared explicitly, but without any details of the proof, as~\cite[equation (3.40)]{hl}.

\begin{prop} \label{prop:disconnection-full}
The inclusion maps $\Lambda(V,W)^{s,\Gamma,\rho}\hookrightarrow\Lambda(V,W)^s$ induce a $G$-equivariant isomorphism
\[
\coprod_{\rho}\, \fM^{\Gamma,\rho}(V,W) \simto \fM(V,W)^\Gamma,
\]
where the domain is a disconnected union over the finite set of equivalence classes of representations $\rho:\Gamma\to\GL(V)$. Hence the nonempty varieties $\fM^{\Gamma,\rho}(V,W)$ constitute the connected components of $\fM(V,W)^\Gamma$.
\end{prop}

\noindent
As we mentioned in~\cite{hl}, one can give an `elementary' proof of Proposition~\ref{prop:disconnection-full}, which is analogous to (but simpler than) the proof of~\cite[Theorem 3.9]{hl}. The first step, analogous to~\cite[Lemma 3.17]{hl}, is to note that the inclusion $\Lambda(V,W)^{s,\Gamma,\rho}\hookrightarrow\Lambda(V,W)^s$ does indeed induce an injection $\fM^{\Gamma,\rho}(V,W)\hookrightarrow\fM(V,W)^\Gamma$, which follows immediately from~\eqref{eqn:stability-freeness}. The next step, analogous to~\cite[Lemma 3.18]{hl}, is to show that every point of $\fM(V,W)^\Gamma$ belongs to the image of $\fM^{\Gamma,\rho}(V,W)\hookrightarrow\fM(V,W)^\Gamma$ for a unique $\rho$ (up to equivalence). Then the rest of the argument is essentially the same as in the proof of~\cite[Theorem 3.9]{hl}: the union is disconnected because the reductive group $\Gamma$ has a discrete classification of representations, and the bijection must be an isomorphism because $\fM(V,W)$ is nonsingular, so every connected component of $\fM(V,W)^\Gamma$ is also nonsingular.

However, as we also mentioned in~\cite{hl}, it is more enlightening to interpret Proposition~\ref{prop:disconnection-full} using moduli spaces. Since~\eqref{eqn:barth-gamma} involved $\fM^{\mathrm{reg}}(V,W)^\Gamma$ rather than the whole of $\fM(V,W)^\Gamma$, we will restrict ourselves to recalling the moduli-space interpretation of the version of Proposition~\ref{prop:disconnection-full} concerning $\fM^{\mathrm{reg}}(V,W)^\Gamma$ (which would be a corollary of the full statement Proposition~\ref{prop:disconnection-full}).

\subsection{$\Gamma$-equivariant vector bundles}
\label{ss:gamma-equiv}

If $(\cE,\Phi)$ is a pair in $\Bun_G(\Aa^2)$, then the second Chern class of $\cE$ is $\dim H^1(\Pp^2,\cE(-\ell_\infty))$. Moreover, the proof of Proposition~\ref{prop:barth} in~\cite[Chapter 2]{nak3} identifies $H^1(\Pp^2,\cE(-\ell_\infty))$ with the vector space $V=\C^n$ in the definition of $\fM(V,W)$. Hence we have a disconnected union
\begin{equation} \label{eqn:disconnection}
\Bun_G^n(\Aa^2/\Gamma)=\coprod_\rho\, \Bun_G^\rho(\Aa^2/\Gamma),
\end{equation} 
where the union is over equivalence classes of representations $\rho:\Gamma\to\GL(V)$. Namely, $\Bun_G^\rho(\Aa^2/\Gamma)$ is the subvariety of $\Bun_G^n(\Aa^2/\Gamma)$ parametrizing $\Gamma$-equivariant pairs $(\cE,\Phi)$ where the induced representation of $\Gamma$ on $V\cong H^1(\Pp^2,\cE(-\ell_\infty))$ is equivalent to $\rho$. As in the above sketched proof of Proposition~\ref{prop:disconnection-full}, the union must be disconnected because $\Gamma$ is reductive.

The following result of Varagnolo and Vasserot is part of~\cite[Theorem 1]{vv}.

\begin{prop} \label{prop:barth-rho}
For each $\rho:\Gamma\to\GL(V)$, the isomorphism $\Theta$ of Proposition~\ref{prop:barth} restricts to an isomorphism 
\[ \fM^{\Gamma,\rho,\mathrm{reg}}(V,W)\simto \Bun_G^\rho(\Aa^2/\Gamma). \]
Here we use the identification of $\fM^{\Gamma,\rho,\mathrm{reg}}(V,W)$ with a subvariety of $\fM^{\mathrm{reg}}(V,W)$ induced by the inclusion map $\Lambda(V,W)^{sc,\Gamma,\rho}\hookrightarrow\Lambda(V,W)^{sc}$. As a consequence, $\Bun_G^\rho(\Aa^2/\Gamma)$ is connected if it is nonempty.
\end{prop} 

Combining~\eqref{eqn:disconnection} and Proposition~\ref{prop:barth-rho}, we get the desired variant of Proposition~\ref{prop:disconnection-full}: 
\begin{equation} \label{eqn:reg-isom}
\fM^{\mathrm{reg}}(V,W)^\Gamma=
\coprod_{\rho}\, \fM^{\Gamma,\rho,\mathrm{reg}}(V,W).
\end{equation}
On the other hand, combining Proposition~\ref{prop:connected} and Proposition~\ref{prop:barth-rho}, we see that in the disconnected union~\eqref{eqn:disconnection}, the nonempty varieties $\Bun_G^\rho(\Aa^2/\Gamma)$ on the right-hand side are exactly the connected components of $\Bun_G^n(\Aa^2/\Gamma)$. Hence the nonempty varieties $\Bun_G^\rho(\Aa^2/\Gamma)$, as not just $\rho$ but also $n$ varies, are the connected components of $\Bun_G(\Aa^2/\Gamma)$.  

As we have already observed in the $\Gamma=\Gg_m$ and $\Gamma=N$ cases (in \S\ref{ss:proof} and \S\ref{ss:n-equiv} respectively), the moduli space $\Bun_G(\Aa^2/\Gamma)$ has another natural decomposition as a disconnected union, where one considers the action of $\Gamma$ on the fibre of the vector bundle at the $\Gamma$-fixed point $[1:0:0]\in\Aa^2$:
\begin{equation} \label{eqn:disconnection2}
\Bun_G(\Aa^2/\Gamma)=\coprod_\tau\, \Bun_G^\tau(\Aa^2/\Gamma).
\end{equation}
Here the union is over $G$-conjugacy classes of homomorphisms $\tau:\Gamma\to G$, i.e.\ equivalence classes of determinant-$1$ representations of $\Gamma$ on $W$ (\emph{not} on $V$ as above). The relationship between~\eqref{eqn:disconnection} and~\eqref{eqn:disconnection2} is pinned down by the next result, which is due to Nakajima.

\begin{prop} \label{prop:fundamental}
Let $\rho:\Gamma\to\GL(V)$ be a representation.
\begin{enumerate}
\item The variety $\fM^{\Gamma,\rho,\mathrm{reg}}(V,W)\cong\Bun_G^\rho(\Aa^2/\Gamma)$ is nonempty if and only if there is a homomorphism $\tau:\Gamma\to G$ such that we have the following equivalence of representations of $\Gamma$:
\begin{equation} \label{eqn:fundamental}
(V,\rho)\oplus (V,\rho) \oplus (W,\tau) \cong \left( (V,\rho)\otimes(\C^2,i) \right) \oplus (W,\mathrm{triv}),
\end{equation}
where $i:\Gamma\hookrightarrow\GL(2)$ is the inclusion and $\mathrm{triv}:\Gamma\to\GL(W)$ is the trivial homomorphism. It is clear that $\tau$ is unique up to $G$-conjugacy if it exists.
\item If $\Bun_G^\rho(\Aa^2/\Gamma)$ is nonempty, then, in the disconnected union~\eqref{eqn:disconnection2}, it is contained in $\Bun_G^\tau(\Aa^2/\Gamma)$ where $\tau$ is as in part \textup{(1)}. 
\end{enumerate}
\end{prop} 

\begin{proof}
Part (1) follows from the nonemptiness criterion given in~\cite[Corollary 10.8]{nak2} (the ``only if'' direction was proved earlier in~\cite[Lemma 8.1]{nak1}), after translating from the setting of general quivers to the specific case of the McKay graph of $\Gamma$. Note that the assumption of~\cite[Proposition 10.5]{nak2} does hold in our setting because $r\geq 2$.
Part (2) amounts to saying that if $(\cE,\Phi)\in\Bun_G(\Aa^2/\Gamma)$, and $\Gamma$ acts via $\rho$ on $V\cong H^1(\Pp^2,\cE(-\ell_\infty))$, then the action of $\Gamma$ on the fibre of $\cE$ at $[1:0:0]\in\Aa^2$ is via the representation $\tau:\Gamma\to G$ defined by~\eqref{eqn:fundamental}. This follows from the description $\cE=\ker(b)/\mathrm{im}(a)$ given in Proposition~\ref{prop:barth}; we need Proposition~\ref{prop:barth-equiv} to identify the $V\oplus V$ in the domain of $b$ with $V\otimes\C^2$ as in~\eqref{eqn:fundamental}.
\end{proof}

As is well known from~\cite{nak1,nak2}, one can translate~\eqref{eqn:fundamental} into very concrete terms. The equivalence class of a representation $\rho:G\to\GL(V)$ may be encoded in the \emph{multiplicity vector} $\bv=(v_i)_{i\in I_\Gamma}$ indexed by the vertices of the McKay graph of $\Gamma$, where $v_i$ is the multiplicity of the irreducible representation $S_i$ in $(V,\rho)$. Note that even if $\Gamma$ is infinite, $\bv$ is finitely-supported. By definition, the multiplicity vector of $(V,\rho)\otimes(\C^2,i)$ is $A_\Gamma\bv$ where $A_\Gamma=(a_{ij})_{i,j\in I_\Gamma}$ is the adjacency matrix of the McKay graph (for the non-simple graphs, we define $A_\Gamma=(2)$ when $\Gamma=\{1\}$ and $A_\Gamma=(\begin{smallmatrix}0&2\\2&0\end{smallmatrix})$ when $|\Gamma|=2$). The multiplicity vector of $(W,\mathrm{triv})$ is $r\delta_0$, where $\delta_0=(\delta_{0i})$ and $0\in I_\Gamma$ is the vertex corresponding to the trivial representation. Define the \emph{Cartan matrix} $C_\Gamma=(c_{ij})_{i,j\in I_\Gamma}$ by $c_{ij}=2\delta_{ij}-a_{ij}$. Then~\eqref{eqn:fundamental} is equivalent to saying that: 
\begin{equation} \label{eqn:fundamental1}
\text{The multiplicity vector of $(W,\tau)$ is $r\delta_0-C_\Gamma \bv$.}
\end{equation} 
The existence of $\tau:\Gamma\to G$ satisfying~\eqref{eqn:fundamental} is therefore equivalent to the condition
\begin{equation} \label{eqn:fundamental2}
r\delta_0-C_\Gamma \bv \geq 0\ \text{ (componentwise nonnegativity).} 
\end{equation} 
Here we have used the fact if $\tau$ is a representation of $\Gamma$ on $W$ satisfying~\eqref{eqn:fundamental}, then automatically $\tau(\Gamma)\subset\SL(W)=G$.  

When $\Gamma$ is finite, $C_\Gamma$ is the Cartan matrix of the corresponding affine Dynkin diagram. It is then well known that $C_\Gamma$ has corank 1, and its kernel is generated by the vector $(\dim S_i)_{i\in I_\Gamma}$ giving the dimensions of the irreducible representations of $\Gamma$. This means that if $(V,\rho)$ gives rise to $(W,\tau)$ as in Proposition~\ref{prop:fundamental}(1), then so does $(V,\rho)\oplus(\C\Gamma,\mathrm{reg})^{\oplus m}$ for any $m\in\N$, where $\mathrm{reg}$ denotes the regular representation of $\Gamma$; but $(W,\tau)$ determines $(V,\rho)$ up to this ambiguity. Thus each nonempty $\Bun_G^\tau(\Aa^2/\Gamma)$ has infinitely many connected components $\Bun_G^\rho(\Aa^2/\Gamma)$, of which, however, there is at most one with any fixed value of the second Chern class $n=\dim V$, as observed in~\cite[Remark 4.6(2)]{bf}. 

When $\Gamma$ is infinite, there is no finitely-supported nonzero vector in the kernel of $C_\Gamma$, so there is at most one $(V,\rho)$ giving rise to a given $(W,\tau)$, and each nonempty $\Bun_G^\tau(\Aa^2/\Gamma)$ is connected. We will now see how the correspondence between $\rho$ and $\tau$ works explicitly in the $\Gamma=\Gg_m$ and $\Gamma=N$ cases.      

\subsection{The $\Gamma=\Gg_m$ case}
\label{ss:Gm-case}

When $\Gamma=\Gg_m$, we continue to label $G$-conjugacy classes of homomorphisms $\tau:\Gg_m\to G$ by dominant coweights $\lambda\in\Lambda^+$, which in the present $G=\SL(r)$ case we identify with weakly decreasing $r$-tuples $(\lambda_1,\lambda_2,\cdots,\lambda_r)\in\Z^r$ such that $\lambda_1+\cdots+\lambda_r=0$. 

As in the previous subsection, we encode an equivalence classes of representations $\rho:\Gg_m\to\GL(V)$ by its multiplicity vector $\bv=(v_i)_{i\in\Z}$. The condition~\eqref{eqn:fundamental2} becomes
\begin{equation} \label{eqn:fund-gm}
r\delta_{0i}-2v_i+v_{i-1}+v_{i+1}\geq 0\text{ for all }i\in\Z,
\end{equation}
and if~\eqref{eqn:fund-gm} holds, the resulting coweight $\lambda=(\lambda_1,\lambda_2,\cdots,\lambda_r)$ is determined by the property that the number $m_i(\lambda)$ of occurrences of a given $i\in\Z$ among $\lambda_1,\cdots,\lambda_r$ equals the left-hand side of~\eqref{eqn:fund-gm}. All dominant coweights arise in this way: for a given $\lambda\in\Lambda^+$, the vector $\bv$ giving rise to it is defined by 
\begin{equation} \label{eqn:v-gm}
v_i=\begin{cases}
\displaystyle\sum_{s=1}^r\max\{\lambda_s-i,0\},&\text{ if $i\geq 0$,}\\
\displaystyle\sum_{s=1}^r\max\{i-\lambda_s,0\},&\text{ if $i<0$.}
\end{cases}
\end{equation}
Of course, $\bv$ is finitely supported because $v_i=0$ unless $\lambda_1>i>\lambda_r$. Note that 
\begin{equation} \label{eqn:n-gm}
n=\sum_{i\in\Z} v_i=\frac{1}{2}\sum_{s=1}^r \lambda_s^2,
\end{equation}
in accordance with~\cite[Theorem 5.2(1)]{bf}. 
  
Hence, in the $\Gamma=\Gg_m$ case, Proposition~\ref{prop:fundamental} says that for any $\lambda\in\Lambda^+$ we have $\Bun_G^\lambda(\Aa^2/\Gg_m)=\Bun_G^\rho(\Aa^2/\Gg_m)$ where $\rho:\Gg_m\to\GL(V)$ is a representation of $\Gg_m$ with multiplicity vector $\bv$ given by~\eqref{eqn:v-gm} (and degree $n$ given by~\eqref{eqn:n-gm}). 

Recall that Proposition~\ref{prop:barth-rho} gives us an isomorphism between $\Bun_G^\rho(\Aa^2/\Gg_m)$ and the quiver variety $\fM^{\Gg_m,\rho,\mathrm{reg}}(V,W)=\Lambda(V,W)^{sc,\Gg_m,\rho}/\GL_{\Gg_m,\rho}(V)$. We now recall how to pass to the more usual description of this as a quiver variety of type $\A$.

The decomposition of $(V,\rho)$ into its isotypic components for $\Gg_m$ constitutes a $\Z$-grading $V=\bigoplus_{i\in\Z} V_i$. Since we are using the trivial representation of $\Gg_m$ on $W$, the analogous $\Z$-grading of $W$ simply has $W_0=W$ and $W_i=\{0\}$ for $i\neq 0$. 

By definition of the $\Gg_m$-action on $\Lambda(V,W)$, a quadruple $(B_1,B_2,\bi,\bj)\in\Lambda(V,W)$ is fixed by $\Gg_m$ if and only if the following conditions are satisfied:
\begin{itemize}
\item $B_1$ and $B_2$ lower and raise the $\Z$-gradings respectively, i.e.\ $B_1(V_i)\subseteq V_{i-1}$ and $B_2(V_i)\subseteq V_{i+1}$ for all $i\in\Z$;
\item $\bi:W\to V$ respects the $\Z$-gradings, i.e.\ $\bi(W)\subseteq V_0$;
\item $\bj:V\to W$ respects the $\Z$-gradings, i.e.\ $\bj(V_i)=0$ unless $i=0$.
\end{itemize}
Thus an element of $\Lambda(V,W)^{\Gg_m,\rho}$ consists of a configuration of linear maps:
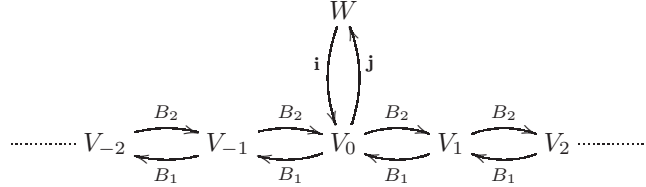
\[
\vcenter{
\xymatrix@R=15pt{
&&&W\ar@/_/[dd]_(.4){\bi}&&&
\\
\\
&V_{-2}\ar@/^/[r]^-{B_{2}}\ar@{.}[l]&V_{-1}\ar@/^/[l]^-{B_{1}}\ar@/^/[r]^-{B_{2}}&V_{0}\ar@/^/[l]^-{B_{1}}\ar@/^/[r]^-{B_{2}}\ar@/_/[uu]_(.6){\bj}&V_{1}\ar@/^/[l]^-{B_{1}}\ar@/^/[r]^-{B_{2}}&V_2\ar@{.}[r]\ar@/^/[l]^-{B_{1}}&
}
}
\]
where we have written simply $B_1$ for each component $B_1|_{V_i}$, and similarly for $B_2$. The equation~\eqref{eqn:adhm}, when broken into its various components, is equivalent to the following equations:
\begin{equation}
\begin{split}
B_1B_2+\bi\bj&=B_2B_1\text{ on }V_0,\\
B_1B_2&=B_2B_1\text{ on $V_i$ for all $i\neq 0$.}
\end{split}
\end{equation} 
Of course, since $V_i=\{0\}$ for sufficiently large $|i|$, the above configuration of linear maps is effectively finite. 

We can obviously identify $\GL_{\Gg_m,\rho}(V)$ with $\prod_{i\in\Z} \GL(V_i)$, and its action on the above configurations of linear maps is the obvious one. So, after making a suitable choice of orientiation of the underlying graph of type $\A$, we recover the usual set-up of quiver varieties of (finite) type $\A$ as in~\cite{nak1,nak2}. In this context it is more common to use the notations $\fM(\bv,\bw)$ for $\fM^{\Gg_m,\rho}(V,W)$ and $\fM_0^{\mathrm{reg}}(\bv,\bw)$ for $\fM^{\Gg_m,\rho,\mathrm{reg}}(V,W)$, where $\bv=(v_i)_{i\in\Z}$ is the dimension vector of the $\Z$-graded vector space $V$, and $\bw$ is defined likewise; in our case, $\bv$ is given by~\eqref{eqn:v-gm} and $\bw=r\delta_0$.   

Combining Proposition~\ref{prop:barth-rho} with Proposition~\ref{prop:identification}, we obtain the following special case of a result of Mirkovi\'c and Vybornov~\cite[Theorem 3.1]{mvy-cr}. Their proof involved a comparison with nilpotent orbits; the moduli-space proof below was suggested in~\cite{bf} and~\cite{nak4}.

\begin{prop} \label{prop:mv}
For any $\lambda\in\Lambda^+$, with $\bv$ defined by~\eqref{eqn:v-gm} and $\bw=r\delta_0$, we have a $G$-equivariant isomorphism
\[
\Psi\circ\Theta:\fM_0^{\mathrm{reg}}(\bv,\bw)\simto\Gr_0^\lambda
\]
which sends the orbit of $(B_1,B_2,\bi,\bj)\in\Lambda(V,W)^{sc,\Gg_m,\rho}$ to
\[
\gamma:=1+(\bj\bi)t^{-1}+(\bj B_2B_1\bi)t^{-2}+(\bj B_2^2B_1^2 \bi)t^{-3}+\cdots \in G[t^{-1}]_1=\Gr_0.
\] 
\end{prop}

\begin{proof}
We have already seen in Proposition~\ref{prop:barth-rho} that we have a $G$-equivariant isomorphism $\Theta:\fM_0^{\mathrm{reg}}(\bv,\bw)=\fM^{\Gg_m,\rho,\mathrm{reg}}(V,W)\simto\Bun_G^\rho(\Aa^2/\Gg_m)=\Bun_G^\lambda(\Aa^2/\Gg_m)$, and Proposition~\ref{prop:identification} gave a $G$-equivariant isomorphism $\Psi:\Bun_G^\lambda(\Aa^2/\Gg_m)\simto\Gr_0^\lambda$. All that remains to be proved is the explicit formula for the composition $\Psi\circ\Theta$. 

To a configuration of maps $(B_1,B_2,\bi,\bj)$ as above that is stable and costable, the isomorphism $\Theta$ assigns a $\Gg_m$-equivariant pair $(\cE,\Phi)\in\Bun_G(\Aa^2)$ by the rule of Proposition~\ref{prop:barth}; that is, $\cE=\ker(b)/\mathrm{im}(a)$ with $\Phi$ being the trivialization on $\ell_\infty$ explained after Proposition~\ref{prop:barth}. We want to show that $\Psi(\cE,\Phi)\in\Gr_0^\lambda$ equals the element $\gamma$ in the statement, and we will use the description of $\Bun_G(\Aa^2)$ and of the bijection $\Psi:\Bun_G(\Aa^2/\Gg_m)\to\Gr_0$ given in \S\ref{ss:p2}.

Consider the restriction of $\cE$ to the open set $U_1$ defined by $z_1\neq 0$. Since $B_1$ lowers the $\Z$-grading on $V$ as seen above, it is nilpotent; hence, on $U_1$, the linear transformation $z_0 B_1 - z_1\mathrm{id}_V$ appearing in the definitions of $a$ and $b$ is invertible. So we can extend $\Phi^{-1}|_{U_1\smallsetminus U_0}$ to a trivialization of $\cE$ on $U_1$, namely the isomorphism $W\otimes\cO_{U_1}\simto \cE|_{U_1}$ induced by the injective map
\[
[0\quad -z_0(z_0 B_1-z_1 \mathrm{id}_V)^{-1}\bi\quad \mathrm{id}_W]:W\otimes\cO_{U_1}\to (V\oplus V\oplus W)\otimes\cO_{U_1},
\]
whose image clearly lies in $\ker(b|_{U_1})$ and is transverse to $\mathrm{im}(a|_{U_1})$.

In the same way, we can extend $\Phi^{-1}|_{U_2\smallsetminus U_0}$ to a trivialization of $\cE$ on $U_2$, namely the isomorphism $W\otimes\cO_{U_2}\simto \cE|_{U_2}$ induced by the injective map
\[
[z_0(z_0 B_2-z_2 \mathrm{id}_V)^{-1}\bi\quad 0\quad \mathrm{id}_W]:W\otimes\cO_{U_2}\to (V\oplus V\oplus W)\otimes\cO_{U_2},
\]
whose image clearly lies in $\ker(b|_{U_2})$ and is transverse to $\mathrm{im}(a|_{U_2})$.

The transition function on $U_1\cap U_2$ relating these trivializations of $\cE$ on $U_1$ and $U_2$ is the following element of $\ker(G[z_0/z_1,z_2/z_1,z_1/z_2]\to G[z_2/z_1,z_1/z_2])$:
\[
\begin{split}
g_2^1&=\mathrm{id}_W+\bj z_0(z_0B_2-z_2\mathrm{id}_V)^{-1}z_0(z_0 B_1-z_1 \mathrm{id}_V)^{-1}\bi\\
&=\mathrm{id}_W+(z_0^2/z_1z_2)\,\bj(\mathrm{id}_V-(z_0/z_2)B_2)^{-1}(\mathrm{id}_V-(z_0/z_1)B_1)^{-1}\bi\\ 
&=\mathrm{id}_W+(z_0^2/z_1z_2)\,\bj\left(\sum_{i,j=0}^\infty (z_0/z_2)^i(z_1/z_2)^j B_2^i B_1^j\right)\bi\\
&=\mathrm{id}_W+(z_0^2/z_1z_2)\,\bj\left(\sum_{i=0}^\infty (z_0/z_2)^i(z_1/z_2)^i B_2^i B_1^i\right)\bi\\
&=\gamma|_{t\mapsto z_1z_2/z_0^2},
\end{split}
\]
where the fourth equality holds because of the known effect of $B_1,B_2,\bi,\bj$ on the $\Z$-gradings.

Hence, as an element of $\Bun_G(\Aa^2)=Z'/H'$, $(\cE,\Phi)$ is the $H'$-orbit of some triple $(g_1^0,g_2^0,g_2^1)\in Z'$ with $g_2^1=\gamma|_{t\mapsto z_1z_2/z_0^2}$. On the other hand, if $\Psi(\cE,\Phi)=\gamma'\in\Gr_0$, then Proposition~\ref{prop:p2-psi} says that $(\cE,\Phi)$ is the $H'$-orbit of a triple $(\tilde{g}_{1}^{0},\tilde{g}_{2}^{0},\tilde{g}_{2}^{1})\in Z'$ with $\tilde{g}_2^1=\gamma'|_{t\mapsto z_1z_2/z_0^2}$. We conclude that
\begin{equation} \label{eqn:endgame}
\gamma'|_{t\mapsto z_1z_2/z_0^2}=h_2\left(\gamma|_{t\mapsto z_1z_2/z_0^2}\right)h_1^{-1}
\end{equation} 
for some elements
\[
\begin{split}
h_1&\in \ker(G[z_0/z_1,z_2/z_1]\to G[z_2/z_1]), \\
h_2 &\in \ker(G[z_0/z_2,z_1/z_2]\to G[z_1/z_2]).
\end{split}
\]

We claim that it follows from~\eqref{eqn:endgame} that $\gamma'=\gamma$ as desired. The reason is that $(\gamma'|_{t\mapsto z_1z_2/z_0^2})h_1(\gamma|_{t\mapsto z_1z_2/z_0^2})^{-1}$ belongs to the group $G[z_0/z_1,z_2/z_1,z_0^2/z_1z_2]$, whose intersection with $\ker(G[z_0/z_2,z_1/z_2]\to G[z_1/z_2])$ is trivial. So~\eqref{eqn:endgame} forces $h_2=1$, and similarly $h_1=1$.
\end{proof}

\begin{rmk} \label{rmk:mv}
The result of Mirkovi\'c and Vybornov~\cite[Theorem 3.1]{mvy-cr} is considerably more general than Proposition~\ref{prop:mv}. Firstly, it involves $\Gr_\mu^\lambda$ for general $\mu$, as in Remark~\ref{rmk:slice}; this corresponds to taking more general choices of $\bw$. Secondly, it relates not just $\fM_0^{\mathrm{reg}}(\bv,\bw)$ with $\Gr_\mu^\lambda$, but the whole affine variety $\fM_0(\bv,\bw)$ with the closure $\Grbar_\mu^\lambda$, and the whole quiver variety $\fM(\bv,\bw)$ with a certain desingularization of $\Grbar_\mu^\lambda$. At least the first of these variants can be proved by an argument similar to the above proof of Proposition~\ref{prop:mv}, using the Uhlenbeck closure of the moduli space $\Bun_G^\lambda(\Aa^2/\Gg_m)$ as in Remark~\ref{rmk:uhlenbeck}; for this, combine~\cite[Theorem 5.2(2)]{bf},~\cite[Theorem 5.12]{bfg} and~\cite[Theorem 1]{vv}.
\end{rmk}

Recall that the involution $\iota$ of $\Gr_0$ satisfies $\iota(\Gr_0^\lambda)=\Gr_0^{-w_0\lambda}$. In terms of $r$-tuples, $-w_0(\lambda_1,\cdots,\lambda_r)=(-\lambda_r,\cdots,-\lambda_1)$. If $\bv$ is the dimension vector corresponding to $\lambda$ as above, then the dimension vector corresponding to $-w_0\lambda$ is $\bv^\dagger$ where $v_i^\dagger=v_{-i}$. Indeed, the representation $\rho^\dagger:\Gg_m\to\GL(V)$ for $-w_0\lambda$ can be chosen to be the composition of $\rho:\Gg_m\to\GL(V)$ with the inverse map on $\Gg_m$, so that the $i$th isotypic component of $V$ for $\rho^\dagger$ is exactly $V_{-i}$.  

\begin{prop} \label{prop:mv-iota}
Let $\lambda\in\Lambda^+$ and define $\bv$ and $\bw$ as above. Under the isomorphisms of Proposition~\ref{prop:mv}, the map $\iota:\Gr_0^\lambda\simto\Gr_0^{-w_0\lambda}$ corresponds to the map $\fM_0^{\mathrm{reg}}(\bv,\bw)\simto\fM_0^{\mathrm{reg}}(\bv^\dagger,\bw)$ which sends the orbit of $(B_1,B_2,\bi,\bj)\in\Lambda(V,W)^{sc,\Gg_m,\rho}$ to the orbit of $(-B_2,B_1,\bi,\bj)\in\Lambda(V,W)^{sc,\Gg_m,\rho^\dagger}$.
\end{prop}

\begin{proof}
Using the formula for the isomorphism in Proposition~\ref{prop:mv}, this claim amounts to saying that for any $(B_1,B_2,\bi,\bj)\in\Lambda(V,W)^{sc,\Gg_m,\rho}$ we have the following equality in $G[t^{-1}]_1$:
\begin{equation}
\left(1-\sum_{i=0}^\infty (\bj B_1^i B_2^i\bi)\,t^{-i-1}\right)^{-1}
=1+\sum_{i=0}^\infty (\bj B_2^i B_1^i\bi)\,t^{-i-1}.
\end{equation}
This equality was proved by direct computation in~\cite[Lemma 4.10]{hl}. But we now have a more conceptual explanation: the claim follows by combining Theorem~\ref{thm:normalizer} with the $\GL(2)$-equivariance of $\Theta$ proved in Proposition~\ref{prop:barth-equiv}. 
\end{proof}

If $\lambda\in\Lambda_1^+$, i.e.\ $\lambda_{r+1-s}=-\lambda_s$ for all $s$, then the corresponding $\bv$ satisfies $\bv^\dagger=\bv$. In this case, Proposition~\ref{prop:mv-iota} says that the involution $\iota$ of $\Gr_0^\lambda$ corresponds under the isomorphism of Proposition~\ref{prop:mv} to a \emph{diagram involution} of $\fM_0^{\mathrm{reg}}(\bv,\bw)$, in the  sense of~\cite[Section 3.2]{hl} (specifically, set $\sigma_n=\mathrm{id}$ in~\cite[Example 3.8]{hl}; one must take into account that in~\cite{hl} we used a slightly different form of the ADHM equation). In the next subsection we will effectively consider the fixed-point subvarieties of these involutions. On the quiver variety side, such a fixed-point subvariety is described in great generality by~\cite[Theorem 3.9]{hl}. However, there will be no need to invoke that result; the reason we can bypass it is that, in the special case of diagram involutions of type-$\A$ quiver varieties, one can deduce~\cite[Theorem 3.9]{hl} easily from Proposition~\ref{prop:disconnection-full}, as outlined in~\cite[Section 3.7]{hl}.     

\subsection{The $\Gamma=N$ case}
\label{ss:N-case}

Recall from \S\ref{ss:n-equiv} the notation $\Xi$ for the set of $G$-conjugacy classes of homomorphisms $\tau:N\to G$, i.e.\ equivalence classes of deteminant-$1$ representations of $N$ on $W=\C^r$, and the disjoint union $\Xi=\bigsqcup_{\lambda\in\Lambda_1^+}\Xi(\lambda)$.

Suppose that $\lambda\in\Lambda_1^+$, and let $\xi\in\Xi(\lambda)$. By definition, we can choose $\tau\in\xi$ such that the restriction of $\tau$ to $\Gg_m$ is $\lambda$. Recall the labelling of the irreducible representations of $N$ given in \S\ref{ss:subgroups}, and note that a representation of $N$ has determinant $1$ if and only if the sum of the multiplicities of $S_{0,-},S_2,S_4,\cdots$ is even. For $i>0$, the multiplicity of the irreducible representation $S_i$ in $(W,\tau)$ is $m_i(\lambda)=m_{-i}(\lambda)$, and the sum of the multiplicities of $S_{0,+}$ and $S_{0,-}$ is $m_0(\lambda)$. Thus, given $\lambda\in\Lambda_1^+$, the choice of $\xi\in\Xi(\lambda)$ is equivalent to the choice of an ordered pair of nonnegative integers $(m_{0,+},m_{0,-})$ satisfying
\begin{equation} \label{eqn:ordered-pair}
m_{0,+}+m_{0,-}=m_0(\lambda),\quad m_{0,-}\equiv m_2(\lambda)+m_4(\lambda)+\cdots\ (\text{mod}\ 2), 
\end{equation}
where $m_{0,\pm}$ specifies the multiplicity of $S_{0,\pm}$.

Let $\rho:N\to\GL(V)$ be a representation, and let $\bv=(v_{0,+},v_{0,-},v_1,v_2,\cdots)$ be its multiplicity vector. For convenience set $v_0:=v_{0,+}+v_{0,-}$. The multiplicity vector of the restriction of $\rho$ to $\Gg_m$ is $\bv^{\A}:=(\cdots,v_2,v_1,v_0,v_1,v_2,\cdots)$. The condition~\eqref{eqn:fundamental2} becomes the conjunction of the following inequalities:
\begin{equation} \label{eqn:fund-n}
\begin{split}
r-2v_{0,+}+v_1&\geq 0,\\
-2v_{0,-}+v_1&\geq 0,\\
-2v_i+v_{i-1}+v_{i+1}&\geq 0\text{ for all }i\geq 1,
\end{split}
\end{equation}
and if~\eqref{eqn:fund-n} is satisfied, the resulting element $\xi\in\Xi$ is the one with multiplicity vector $(m_{0,+},m_{0,-},m_1,m_2,\cdots)$ given by the left-hand sides of~\eqref{eqn:fund-n}. Note that 
\[ m_{0,-}+m_2+m_4+\cdots =  -2v_{0,+}+2v_1-2v_2+2v_3-2v_4+\cdots \]
is indeed even. Set $m_0:=m_{0,+}+m_{0,-}=r-2v_0+2v_1$. Thus $\xi\in\Xi(\lambda)$ where $\lambda\in\Lambda_1^+$ has multiplicity vector $(\cdots,m_2,m_1,m_0,m_1,m_2,\cdots)$. Clearly this coweight $\lambda$ is the one associated to $\bv^{\A}$ as in \S\ref{ss:Gm-case}.

Conversely, given $\lambda\in\Lambda_1^+$, we can define $\bv^{\A}$ by~\eqref{eqn:v-gm}, and it will automatically have the symmetric form $\bv^{\A}=(\cdots,v_2,v_1,v_0,v_1,v_2,\cdots)$. Note that $v_0$ is the sum of the positive entries in the $r$-tuple $\lambda$, and $v_0-v_1$ is the number of these positive entries, implying that $v_1\equiv m_2(\lambda)+m_4(\lambda)+\cdots\ (\text{mod}\ 2)$. Suppose that $\xi\in\Xi(\lambda)$ corresponds to the ordered pair $(m_{0,+},m_{0,-})$ satisfying~\eqref{eqn:ordered-pair}; in terms of $\bv^{\A}$, this latter condition becomes
\begin{equation} \label{eqn:ordered-pair2}
m_{0,+}+m_{0,-}=r-2v_0+2v_1,\quad m_{0,-}\equiv v_1\ (\text{mod}\ 2).
\end{equation}
Then $\xi$ arises from some representation $\rho:N\to\GL(V)$ if and only if we can write
\begin{equation} \label{eqn:prelim}
m_{0,+}=r-2v_{0,+}+v_1,\quad m_{0,-}=-2v_{0,-}+v_1
\end{equation}
for some nonnegative integers $v_{0,+},v_{0,-}$ (whose sum is then necessarily $v_0$). Rearranging~\eqref{eqn:prelim}, we see that this condition is equivalent to requiring that the following are nonnegative integers:
\begin{equation} \label{eqn:v-N}
v_{0,+}:=\frac{1}{2}(r-m_{0,+}+v_1),\quad v_{0,-}:=\frac{1}{2}(-m_{0,-}+v_1),
\end{equation}
and this in turn, in view of~\eqref{eqn:ordered-pair2}, is equivalent to requiring simply that 
$m_{0,-}\leq v_1$.

To sum up, if $m_{0,-}\leq v_1$, 
then Proposition~\ref{prop:fundamental} tells us that $\Bun_G^\xi(\Aa^2/N)=\Bun_G^\rho(\Aa^2/N)$ is nonempty, and Proposition~\ref{prop:barth-rho} tells us that it is connected. If $m_{0,-}>v_1$, 
then Proposition~\ref{prop:fundamental} tells us that $\Bun_G^\xi(\Aa^2/N)$ is empty.

Recall that Proposition~\ref{prop:barth-rho} gives us an isomorphism between $\Bun_G^\rho(\Aa^2/N)$ and $\fM^{N,\rho,\mathrm{reg}}(V,W)=\Lambda(V,W)^{sc,N,\rho}/\GL_{N,\rho}(V)$. We will now see how to express the latter variety as a quiver variety of type $\D$ in the usual sense. The idea is much the same as in the type-$\A$ case considered in \S\ref{ss:Gm-case}.

The decomposition of $(V,\rho)$ into its isotypic components for $N$ is closely related to the $\Z$-grading $V=\bigoplus_{i\in\Z} V_i$ obtained from the restriction of $\rho$ to $\Gg_m$. Namely, we have a direct sum decomposition $V_0=V_{0,+}\oplus V_{0,-}$ into isotypic components corresponding to the irreducible representations $S_{0,+}$ and $S_{0,-}$, and the isotypic component corresponding to $S_i$ for $i>0$ is $V_i\oplus V_{-i}\cong V_i\otimes S_i$. Note that $\rho([\begin{smallmatrix}0&1\\-1&0\end{smallmatrix}])$ maps $V_i$ isomorphically onto $V_{-i}$ for all $i\neq 0$, and acts as the identity on $V_{0,+}$ and as minus the identity on $V_{0,-}$. Since we are using the trivial representation of $N$ on $W$, the analogous decomposition of $W$ has only one nonzero term, $W=W_{0,+}$. 

A quadruple $(B_1,B_2,\bi,\bj)\in\Lambda(V,W)$ is fixed by $N$ if and only if it is fixed by $\Gg_m$ and also fixed by $[\begin{smallmatrix}0&1\\-1&0\end{smallmatrix}]$. In \S\ref{ss:Gm-case} we have already seen how to translate the condition of being fixed by $\Gg_m$ in terms of the $\Z$-gradings on $V$ and $W$. By definition of the $N$-action on $\Lambda(V,W)$, the additional condition of being fixed by $[\begin{smallmatrix}0&1\\-1&0\end{smallmatrix}]$ amounts to the following extra constraints:
\begin{itemize}
\item $B_2=\rho([\begin{smallmatrix}0&1\\-1&0\end{smallmatrix}])B_1\rho([\begin{smallmatrix}0&1\\-1&0\end{smallmatrix}])^{-1}$, which means that the maps $B_1|_{V_i}$ for $i\leq 0$ are uniquely determined by the maps $B_2|_{V_i}$ for $i\geq 0$ and the maps $B_2|_{V_i}$ for $i<0$ are uniquely determined by the maps $B_1|_{V_i}$ for $i>0$;
\item $\bi(W)\subseteq V_{0,+}$;
\item $\bj(V_{0,-})=0$.
\end{itemize}
Thus an element of $\Lambda(V,W)^{N,\rho}$, after removing the redundant data, consists of a configuration of linear maps of the form
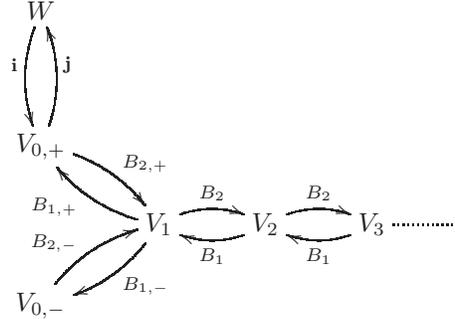
\[
\vcenter{
\xymatrix@R=15pt{
W\ar@/_/[dd]_(.4){\bi}&&&&
\\
\\
V_{0,+}\ar@/^/[dr]^(.6){B_{2,+}}\ar@/_/[uu]_(.6){\bj}&&&&\\
&V_{1}\ar@/^/[ul]^(.6){B_{1,+}}\ar@/^/[dl]^(.4){B_{1,-}}\ar@/^/[r]^-{B_{2}}&V_2\ar@/^/[l]^-{B_{1}}\ar@/^/[r]^-{B_{2}}&V_3\ar@{.}[r]\ar@/^/[l]^-{B_{1}}&\\
V_{0,-}\ar@/^/[ur]^(.4){B_{2,-}}&&&&
}
}
\]
where $B_{1,+}$ and $B_{1,-}$ are the components of $B_1|_{V_1}$, and $B_{2,+}$ and $B_{2,-}$ are the components of $B_2|_{V_0}$, relative to the direct sum decomposition $V_0=V_{0,+}\oplus V_{0,-}$. The equation~\eqref{eqn:adhm}, when broken into its various components and simplified using the rule $B_2=\rho([\begin{smallmatrix}0&1\\-1&0\end{smallmatrix}])B_1\rho([\begin{smallmatrix}0&1\\-1&0\end{smallmatrix}])^{-1}$, is equivalent to the following equations:
\begin{equation} \label{eqn:type-d}
\begin{split}
2B_{1,+}B_{2,+}+\bi\bj&=0\text{ on $V_{0,+}$,}\\
B_{1,-}B_{2,-}&=0\text{ on $V_{0,-}$,}\\ 
B_1B_2&=B_{2,+}B_{1,+}+B_{2,-}B_{1,-}\text{ on $V_1$,}\\
B_1B_2&=B_2B_1\text{ on $V_i$ for all $i>1$.} 
\end{split}
\end{equation}
Again, since $V_i=\{0\}$ for sufficiently large $i$, the above configuration of maps is effectively finite.

We can obviously identify $\GL_{N,\rho}(V)$ with $\GL(V_{0,+})\times\GL(V_{0,-})\times\prod_{i>0}\GL(V_i)$, and its action on the above configurations of linear maps is the obvious one. So, after making a suitable choice of orientiation of the underlying graph of type $\D$, we recover the usual set-up of quiver varieties of (finite) type $\D$ as in~\cite{nak1,nak2}. (We can account for the factor of $2$ in the first equation of~\eqref{eqn:type-d} by scaling $\bi$, as in~\cite[Example 3.12]{hl}.) In this context it is more common to use the notations $\fM(\bv,\bw)$ for $\fM^{N,\rho}(V,W)$ and $\fM_0^{\mathrm{reg}}(\bv,\bw)$ for $\fM^{N,\rho,\mathrm{reg}}(V,W)$, where $\bv=(v_{0,+},v_{0,-},v_1,v_2,\cdots)$ is the vector of dimensions of $V_{0,+},V_{0,-},V_1,V_2,\cdots$ and $\bw$ is defined likewise; in our case, $\bv$ is given by~\eqref{eqn:v-gm} and~\eqref{eqn:v-N}, and $\bw=(r,0,0,0,\cdots)$.   

We can now state a more specific version of Theorem~\ref{thm:intro-glr}:

\begin{prop} \label{prop:glr}
Let $\lambda=(\lambda_1,\cdots,\lambda_r)\in\Lambda_1^+$ and $\xi\in\Xi(\lambda)$. Then $(\Gr_0)^{\iota,\xi}$ is nonempty if and only if the pair $(m_{0,+},m_{0,-})$ determined by $\xi$, which by definition satisfies~\eqref{eqn:ordered-pair}, also satisfies
\[ m_{0,-}\leq \sum_{s=1}^r \max\{\lambda_s-1,0\}. \] 
Assume henceforth that this condition holds. Define $\bv=(v_{0,+},v_{0,-},v_1,v_2,\cdots)$ by~\eqref{eqn:v-gm} and~\eqref{eqn:v-N} and set $\bw=(r,0,0,0,\cdots)$. Then we have a $G$-equivariant isomorphism
\[
\Psi\circ\Theta:\fM_0^{\mathrm{reg}}(\bv,\bw)\simto(\Gr_0)^{\iota,\xi}
\]
which sends the orbit of a configuration of linear maps as above to the following element of $G[t^{-1}]_1$:
\[
1+(\bj\bi)t^{-1}+\sum_{i=0}^\infty (-1)^{i+1} (\bj B_{1,+}B_1^i B_2^iB_{2,+}\bi)\,t^{-i-2}.
\]
In particular, $(\Gr_0)^{\iota,\xi}$ is nonsingular and connected when it is nonempty.
\end{prop}

\begin{proof}
Recall that the $G$-equivariant bijection $\Psi:\Bun_G(\Aa^2/N)\to(\Gr_0)^\iota$ restricts to an isomorphism $\Bun_G^\xi(\Aa^2/N)\simto(\Gr_0)^{\iota,\xi}$. So the nonemptiness criterion for $(\Gr_0)^{\iota,\xi}$ follows from that for $\Bun_G^\xi(\Aa^2/N)$, seen above. Assume that these varieties are indeed nonempty. Proposition~\ref{prop:barth-rho} tells us that the $G$-equivariant isomorphism $\Theta:\fM^{\mathrm{reg}}(V,W)\simto\Bun_G^n(\Aa^2)$ restricts to an isomorphism $\fM_0^{\mathrm{reg}}(\bv,\bw)=\fM^{N,\rho,\mathrm{reg}}(V,W)\simto\Bun_G^\rho(\Aa^2/N)=\Bun_G^\xi(\Aa^2/N)$. Recall that $\fM^{N,\rho,\mathrm{reg}}(V,W)$ is nonsingular and connected by Proposition~\ref{prop:connected}.

All that remains to be proved is that the formula for the composition $\Psi\circ\Theta$ on $\fM^{\Gg_m,\rho|_{\Gg_m},\mathrm{reg}}(V,W)$ given in Proposition~\ref{prop:mv} restricts to the stated formula on $\fM_0^{\mathrm{reg}}(\bv,\bw)=\fM^{N,\rho,\mathrm{reg}}(V,W)$. This follows from the fact that, for any quadruple $(B_1,B_2,\bi,\bj)\in\Lambda(V,W)^{N,\rho}$ and any $i\geq 0$,
\[
\begin{split}
\bj B_2^{i+1} B_1^{i+1} \bi
&=\bj \left(\rho([\begin{smallmatrix}0&1\\-1&0\end{smallmatrix}])B_1\rho([\begin{smallmatrix}0&1\\-1&0\end{smallmatrix}])^{-1}\right)^{i+1} \left(-\rho([\begin{smallmatrix}0&1\\-1&0\end{smallmatrix}])B_2\rho([\begin{smallmatrix}0&1\\-1&0\end{smallmatrix}])^{-1}\right)^{i+1} \bi\\
&=(-1)^{i+1} \bj B_1^{i+1}B_2^{i+1} \bi=(-1)^{i+1} \bj B_{1,+} B_1^i B_2^i B_{2,+} \bi.
\end{split}
\]  
Here the second equality holds because $\rho([\begin{smallmatrix}0&1\\-1&0\end{smallmatrix}])$ acts as the identity on $V_{0,+}$.
\end{proof}

\begin{ex} \label{ex:disconnected}
Take $r=4$ and $\lambda=(3,0,0,-3)\in\Lambda_1^+$. The corresponding $\Z$-tuple $\bv^{\A}$ is $(\cdots,0,0,1,2,3,2,1,0,0,\cdots)$ with $v_0=3$. The relevant ordered pairs $(m_{0,+},m_{0,-})$ are $(2,0)$ and $(0,2)$; let $\xi_1$ and $\xi_2$ be the corresponding elements of $\Xi(\lambda)$. Then $(\Gr_0)^{\iota,\xi_1}$ and $(\Gr_0)^{\iota,\xi_2}$ are the connected components of $(\Gr_0^\lambda)^\iota$, and by Proposition~\ref{prop:glr} they are isomorphic to the following quiver varieties respectively:
\[
\fM_0^{\mathrm{reg}}((2,1,2,1),(4,0,0,0))\text{ and }
\fM_0^{\mathrm{reg}}((3,0,2,1),(4,0,0,0)),
\] 
where we have put the vertices of $\D_\infty$ in the order $(0,+),(0,-),1,2,3,\cdots$ and truncated all vertices where $v_i=0$, leaving quiver varieties of type $\D_4$ (except that the second is really a quiver variety of type $\A_3$, since it happens to have $v_{0,-}=0$). 
\end{ex}

\begin{rmk}
A general formula for the dimension of a Nakajima quiver variety is given in~\cite[Corollary 3.12]{nak2}. Applied to the quiver variety of type $\A$ in Proposition~\ref{prop:mv}, this recovers the well-known formula $\dim\Gr_0^\lambda=\langle\lambda,2\rho\rangle$ for any $\lambda\in\Lambda^+$, where $2\rho$ is the sum of the positive roots of $\SL(r)$. Applied to the quiver variety of type $\D$ in Proposition~\ref{prop:glr}, it gives
\begin{equation}
\dim (\Gr_0)^{\iota,\xi} = \langle\lambda,\rho\rangle + \frac{1}{4}(r^2-(m_{0,+}-m_{0,-})^2),
\end{equation}
where $\lambda\in\Lambda_1^+$ and $\xi\in\Xi(\lambda)$ corresponds to the pair $(m_{0,+},m_{0,-})$. 
\end{rmk}

\subsection{Proof of Theorem~\ref{thm:intro-gl2}}
\label{ss:proof-gl2}

Now suppose that $G=\SL(2)$ and that $\lambda=m\alpha=(m,-m)$ for $m\in\Z^+$. The corresponding $\Z$-tuple $\bv^{\A}$ is 
\[
(\cdots,0,0,1,2,\cdots,m-1,m,m-1,\cdots,2,1,0,0,\cdots)
\] 
with $v_0=m$. Hence Proposition~\ref{prop:mv} says that 
\begin{equation} \label{eqn:special-mv}
\Gr_0^{m\alpha}\cong\fM_0^{\mathrm{reg}}(\bv^{\A},2\delta_0),
\end{equation}
a quiver variety of type $\A_{2m-1}$. 

There is another isomorphism involving this quiver variety:
\begin{equation} \label{eqn:kronheimer}
\fM_0^{\mathrm{reg}}(\bv^{\A},2\delta_0)\cong\cO_{(2m)}\cap\cS_{(m,m)},
\end{equation}
where $\cO_{(2m)}$ is the regular nilpotent orbit in $\fsl(2m)$ and $\cS_{(m,m)}$ is the Slodowy slice to the orbit $\cO_{(m,m)}$. The isomorphism~\eqref{eqn:kronheimer} is a special case of Nakajima's result~\cite[Theorem 8.4]{nak1}, a reformulation of Kronheimer's result~\cite[Theorem 1]{kron}. Thus, it implicitly treats the graph of type $\A_{2m-1}$ as part of the McKay graph of $\SL(2)$, \emph{not} the McKay graph of $\Gg_m$ as in the proof of~\eqref{eqn:special-mv}; this distinction was highlighted in~\cite[Section 2(v)]{nak4}. Philosophically, this is why we obtain an interesting isomorphism by composing~\eqref{eqn:special-mv} and~\eqref{eqn:kronheimer}, namely
\begin{equation} \label{eqn:tworow}
\Gr_0^{m\alpha}\cong \cO_{(2m)}\cap\cS_{(m,m)}.
\end{equation}

\begin{rmk}
We do not know an explicit formula for the isomorphism~\eqref{eqn:kronheimer}. Instead, there is a recursive procedure for passing from a quadruple $(B_1,B_2,\bi,\bj)$ to an element of $\fsl(2m)$, which is due to Maffei~\cite[Theorem 8]{maffei} in greater generality, and is rephrased in~\cite[Lemma 4.8]{hl} in the case relevant for~\eqref{eqn:kronheimer}. Note that Mirkovi\'c and Vybornov~\cite[Section 3.2]{mvy-cr} provided explicit versions of~\eqref{eqn:kronheimer} and~\eqref{eqn:tworow} where the Slodowy slice $\cS_{(m,m)}$ is replaced by a different transverse slice to the orbit $\cO_{(m,m)}$. However, their transverse slice is not stable under the negative-transpose involution considered below.
\end{rmk} 

As a special case of Proposition~\ref{prop:mv-iota}, the involution $\iota$ of $\Gr_0^{m\alpha}$ corresponds under the isomorphism~\eqref{eqn:special-mv} to a diagram involution of $\fM_0^{\mathrm{reg}}(\bv^{\A},2\delta_0)$. As a special case of~\cite[Theorem 4.4]{hl} (and this is where the assumption $r=2$ is vital), this diagram involution of $\fM_0^{\mathrm{reg}}(\bv^{\A},2\delta_0)$ corresponds, under the isomorphism of~\eqref{eqn:kronheimer}, to an involution of $\cO_{(2m)}\cap\cS_{(m,m)}$ that is the restriction of a Lie algebra involution of $\fsl(2m)$, namely the negative transpose map with respect to a nondegenerate form on $\C^{2m}$ that is symmetric if $m$ is even and skew-symmetric if $m$ is odd. See~\cite[Section 4]{hl} for the detailed definitions.

We now consider the fixed-point subvarieties of these involutions. From~\eqref{eqn:ordered-pair} we see that $\Xi(m\alpha)$ is empty if $m$ is even, and has a unique element, corresponding to the ordered pair $(0,0)$, if $m$ is odd. Recall that we already showed that $(\Gr_0^{m\alpha})^\iota$ is empty for $m$ even and positive in Lemma~\ref{lem:even-empty}; on the other side of~\eqref{eqn:tworow}, it is well known that $\cO_{(2m)}$ does not intersect $\fso(2m)$. 

Assume henceforth that $m$ is odd. By Proposition~\ref{prop:glr}, taking fixed points on both sides of ~\eqref{eqn:special-mv} gives the isomorphism
\begin{equation} \label{eqn:explicit}
(\Gr_0^{m\alpha})^\iota\cong\fM_0^{\mathrm{reg}}((\tfrac{m+1}{2},\tfrac{m-1}{2},m-1,\cdots,2,1),(2,0,\cdots,0)),
\end{equation}
where the quiver variety is of type $\D_{m+1}$, with the vertices labelled in the same order as in Example~\ref{ex:disconnected}. (In the $m=1$ case, we have to interpret $\D_2$ as $\A_1\times\A_1$.) Taking fixed points on both sides of~\eqref{eqn:kronheimer} gives the isomorphism
\begin{equation} \label{eqn:implicit}
\fM_0^{\mathrm{reg}}((\tfrac{m+1}{2},\tfrac{m-1}{2},m-1,\cdots,2,1),(2,0,\cdots,0))\cong \cO_{(2m)}^{\mathrm{C}_m}\cap\cS_{(m,m)}^{\mathrm{C}_m},
\end{equation}   
which is a special case of~\cite[Theorem 1.2]{hl}. Finally, Theorem~\ref{thm:intro-gl2} follows by combining~\eqref{eqn:explicit} and~\eqref{eqn:implicit}, or in other words by taking fixed points on both sides of~\eqref{eqn:tworow}.

\begin{rmk}
The isomorphisms~\eqref{eqn:special-mv} and~\eqref{eqn:kronheimer} extend to isomorphisms between the natural closures of these varieties:
\begin{equation} \label{eqn:closures}
\Grbar_0^{m\alpha}\cong\fM_0(\bv^{\A},2\delta_0)\cong\cN\cap\cS_{(m,m)},
\end{equation}
where $\cN$ is the nilpotent cone of $\fsl(2m)$. The first isomorphism in~\eqref{eqn:closures} follows from~\cite[Theorem 5.2]{bf} and~\cite[Theorem 5.12]{bfg}, and the second isomorphism in~\eqref{eqn:closures} is a special case of~\cite[Theorem 8]{maffei}. The involutions of the open subvarieties considered above all extend to the closures. If $m$ is odd, we obtain the following by taking fixed points throughout~\eqref{eqn:closures}:
\begin{equation}
(\Grbar_0^{m\alpha})^\iota\cong\fM_0((\tfrac{m+1}{2},\tfrac{m-1}{2},m-1,\cdots,2,1),(2,0,\cdots,0))\cong\cN^{\mathrm{C}_m}\cap\cS_{(m,m)}^{\mathrm{C}_m},
\end{equation}
where the quiver variety is again of type $\D_{m+1}$ and $\cN^{\mathrm{C}_m}$ is the nilpotent cone of $\fsp(2m)$. For the claim about the fixed-point subvariety of $\fM_0(\bv^{\A},2\delta_0)$, see the proof of~\cite[Theorem 1.2]{hl}. 
\end{rmk}



\begin{thebibliography}{99}


\bibitem{ah}
P.~N.~Achar and A.~Henderson, \emph{Geometric Satake, Springer correspondence, and small representations}, Selecta Math.\ (N.S.) {\bf 19} (2013), no.~4, 949--986.

\bibitem{atiyah}
M.~F.~Atiyah, \emph{Instantons in two and four dimensions}, Comm.\ Math.\ Phys.\ \textbf{93} (1984), no.~4, 437--451.

\bibitem{ahdm}
M.~F.~Atiyah, N.~J.~Hitchin, V.~G.~Drinfel'd and Yu.~I.~Manin, \emph{Construction of instantons}, Phys.\ Lett.\ A \textbf{65} (1978), no.~3, 185--187.

\bibitem{barth}
W.~Barth, \emph{Moduli of vector bundles on the projective plane}, Invent.\ Math.\ \textbf{42} (1977), 63--91. 

\bibitem{bf}
A.~Braverman and M.~Finkelberg, {\em Pursuing the double affine Grassmannian, I: Transversal slices via instantons on $A_k$-singularities}, Duke Math.\ J.\ {\bf 152} (2010), no.~2, 175--206.

\bibitem{bfg}
A.~Braverman, M.~Finkelberg and D.~Gaitsgory, \emph{Uhlenbeck spaces via affine Lie algebras}, in \emph{The unity of mathematics}, Progr.\ Math.\ {\bf 244}, Birkh\"auser Boston, Boston, MA, 2006, 17--135. 

\bibitem{c-b}
W.~Crawley--Boevey, \emph{Geometry of the moment map for representations of quivers}, Compositio Math.\ \textbf{126} (2001), 257--293.

\bibitem{donaldson}
S.~K.~Donaldson, \emph{Instantons and geometric invariant theory}, Comm.\ Math.\ Phys. \textbf{93} (1984), no.~4, 453--460.

\bibitem{furutahashimoto}
M.~Furuta and Y.~Hashimoto, \emph{Invariant instantons on $S^4$}, J.\ Fac.\ Sci.\ Univ.\ Tokyo Sect.\ IA Math.\ \textbf{37} (1990), no.~3, 585--600.

\bibitem{hl}
A.~Henderson and A.~Licata, \emph{Diagram automorphisms of quiver varieties}, Adv.\ Math.\ \textbf{267} (2014), 225--276.

\bibitem{kwwy}
J.~Kamnitzer, B.~Webster, A.~Weekes and O.~Yacobi, \emph{Yangians and quantizations of slices in the affine Grassmannian}, Algebra Number Theory \textbf{8} (2014), no.~4, 857--893.

\bibitem{kron}
P.~B.~Kronheimer, \emph{Instantons and the geometry of the nilpotent variety}, J.\ Differential Geom.\ {\bf 32} (1990), no.~2, 473--490.

\bibitem{kumar}
S.~Kumar, \emph{An approach towards the Koll\'ar--Peskine problem via the instanton moduli space}, in \emph{Recent developments in Lie algebras, groups and representation theory}, Proc.\ Sympos.\ Pure Math.\ \textbf{86}, Amer.\ Math.\ Soc., Providence, 2012, 217--225.

\bibitem{laszlosorger}
Y.~Laszlo and C.~Sorger, \emph{The line bundles on the moduli of parabolic $G$-bundles over curves and their sections}, Ann.\ Sci.\ \'Ecole Norm.\ Sup.\ (4) \textbf{30} (1997), no.~4, 499--525.

\bibitem{lusztig}
G.~Lusztig,  {\it Green polynomials and singularities of unipotent classes},
Adv.\ in Math.\ {\bf 42} (1981), no.~2, 169--178.

\bibitem{maffei}
A.~Maffei, \emph{Quiver varieties of type A}, Comment.\ Math.\ Helv.\ \textbf{80} (2005), no.~1, 1--27. 

\bibitem{mckay}
J.~McKay, \emph{Graphs, singularities, and finite groups}, in \emph{The Santa Cruz Conference on Finite Groups (Univ.\ California, Santa Cruz, Calif., 1979)},
Proc.\ Sympos.\ Pure Math.\ \textbf{37}, Amer.\ Math.\ Soc., Providence, 1980, 183--186.

\bibitem{mv}
I.~Mirkovi{\'c} and K.~Vilonen, \emph{Geometric Langlands duality and representations of algebraic groups over commutative rings}, Ann.~of Math.~(2) \textbf{166} (2007), 95--143.

\bibitem{mvy-cr}
I.~Mirkovi{\'c} and M.~Vybornov, \emph{On quiver varieties and affine
Grassmannians of type A}, C.~R.~Math.\ Acad.\ Sci.\ Paris \textbf{336} (2003), no.~3, 207--212.

\bibitem{nak1}
H.~Nakajima, \emph{Instantons on ALE spaces, quiver varieties, and Kac--Moody algebras},
Duke Math.\ J.\ \textbf{76} (1994), no.~2, 365--416. 

\bibitem{nak2}
H.~Nakajima, \emph{Quiver varieties and Kac--Moody algebras},
Duke Math.\ J.\ \textbf{91} (1998), no.~3, 515--560.

\bibitem{nak3}
H.~Nakajima, \emph{Lectures on Hilbert schemes of points on surfaces}, University Lecture Series {\bf 18}, American Mathematical Society, Providence, RI, 1999.

\bibitem{nak4}
H.~Nakajima, \emph{Towards a mathematical definition of Coulomb branches of $3$-dimensional $\mathcal{N}=4$ gauge theories, I}, arXiv:1503.0367.

\bibitem{slod}
P.~Slodowy, \emph{Simple Singularities and Simple Algebraic Groups}, Lecture Notes in Math., vol.~815, Springer-Verlag, 1980.

\bibitem{vv}
M.~Varagnolo and E.~Vasserot, \emph{On the $K$-theory of the cyclic quiver variety}, Internat.\ Math.\ Res.\ Notices \textbf{1999}, no.~18, 1005--1028. 

\end{thebibliography}
\end{document}